\documentclass{amsart}
\usepackage{amsfonts,amsthm,amsmath}
\usepackage{amssymb}
\usepackage{amscd}
\usepackage[utf8]{inputenc}
\usepackage[a4paper]{geometry}
\usepackage{enumerate}
\usepackage[usenames,dvipsnames]{color}
\usepackage{tikz}
\usepackage{soul}
\usepackage{array}
\usepackage{array,tabularx}

\DeclareMathOperator{\Coeff}{Coeff}
\DeclareMathOperator*{\Res}{Res}
\newcommand\note[1]{\mbox{}\marginpar{ \scriptsize\raggedright
\hspace{1pt}\color{red} #1}}

\numberwithin{equation}{section}
\numberwithin{equation}{subsection}

\theoremstyle{plain}
\newtheorem*{theorem*}{Theorem}

\newtheorem{theorem}[equation]{Theorem}
\newtheorem{lemma}[equation]{Lemma}

\newtheorem{thm}[equation]{Theorem}

\theoremstyle{definition}
\newtheorem{example}[equation]{Example}
\newtheorem{remark}[equation]{Remark}

\newtheorem{definition}[equation]{Definition}

\newtheorem{bekezdes}[equation]{}

\newcommand{\fr}{\mathfrak{r}}

\def\C{\mathbb C}
\def\Q{\mathbb Q}

\def\Z{\mathbb Z}

\def\im{{\rm Im}}

\newcommand{\calv}{{\mathcal V}}

\newcommand{\calm}{{\mathcal M}}

\newcommand{\cali}{{\mathcal I}}

\newcommand{\calO}{{\mathcal O}}

\newcommand{\calS}{{\mathcal S}}

\newcommand{\calL}{\mathcal{L}}

\newcommand{\tX}{\widetilde{X}}
\newcommand{\mfl}{\mathfrak{L}}

\newcommand{\cX}{{\mathcal X}}
\newcommand{\cO}{{\mathcal O}}
\newcommand{\bP}{{\mathbb P}}
\newcommand*{\linebundle}{\mathcal{L}}
\newcommand{\bC}{{\mathbb C}}
\newcommand{\cF}{{\mathcal F}}

\newcommand{\eca}{{\rm ECa}}
\newcommand{\pic}{{\rm Pic}}

\newcommand{\bt}{{\mathbf t}}

\newcommand{\bZ}{{\mathbb{Z}}}
\newcommand{\bQ}{{\mathbb{Q}}}

\author{J\'anos Nagy}
\address{Central European University, Dept. of Mathematics,  Budapest, Hungary}
\email{nagy\textunderscore janos@phd.ceu.edu}

\title{Class of images of Abel maps on normal surface singularities}

\begin{document}



\begin{abstract}
In this paper we investigate Abel maps on normal surface singularities described in \cite{NNI}.
We investigate the affine version of the class of the images of Abel maps on normal surface singularities.
More precisely we consider the projective clousure of the image of an Abel map, its dual projective variety and we substract from its degree the multiplicity of the infinite hyperplane on the dual variety. In the case of generic singularities we prove explicit combinatorial formulas of this invariant, in the general case we prove an upper bound.
\end{abstract}

\maketitle

\linespread{1.2}


\pagestyle{myheadings} \markboth{{\normalsize  J. Nagy}} {{}}


\section{Introduction}

In this paper we investigate Abel maps on normal surface singularities described in \cite{NNI}, which were useful in studying invariants like multiplicity or geometric genus of generic analytic structures of normal surface singularities in \cite{NNII} and \cite{NNM}.

In \cite{NNAD} the author and A. Némethi studied the image varieties of Abel maps in the corresponding Picard groups focusing mostly on the dimension of these varieties, and computed these dimensions algorithmically from analytic invariants of the singularity, like cohomology numbers of cycles getting also explicit combinatorial formulae from the resolution graph, when the analytic type is generic.

The reason of our interest in these image varieties is that these are irreducible components of Brill-Noether stratas in the corresponding Picard groups with the value of $h^1$ equal to its
codimension (see \cite{NNAD}).

In the classical case of smooth curves the computation of dimensions of Brill Noether stratas is also a cruical problem, however the dimension of images of Abel maps is 
a special case of it and one can easily see, that if $d \geq 0$ and $C$ is a smooth curve of genus $g$, then the dimension of the Abel map $Symm^d(C) \to \pic^d(C)$ is $\min(d, g)$.
In the case of normal surface singularities these are already intresting invariants which vary if we move the analytic type of the singularity.

In this paper we investigate the affine version of the class of the images of Abel maps on normal surface singularities.
More precisely we consider the projective clousure of the image of an Abel map, its dual projective variety, and we substract from its degree the multiplicity of the infinite hyperplane on the dual variety (it is $0$ if the infinite hyperplane is not on the dual variety), we denote this invariant by $\tau$ throughout the paper.

In the case of generic singularities we prove the following main theorem (the technical condition $Z = C_{min}(Z, l')$ is explained later):

\begin{theorem*}
Let $\mathcal{T}$ be an arbitrary resolution graph and $\tX$ a generic singularity corresponding to it.
Let's have a Chern class $l' \in -S'$ and an integer effective cycle $Z \geq E$, such that $Z = C_{min}(Z, l')$, notice that this is 
a combinatorial condition computable from the resolution graph if the singularity is generic, and in particular we know that the map $ \eca^{l'}(Z) \to \im(c^{l'}(Z))$ is birational.
With these notations we have the following:

1) The dual projective variety of the projective clousure $\overline{\im(c^{l'}(Z))}$ has got dimension $h^1(\calO_Z)-1$.

2) Let's have the line bundle $\calL_Z = \calO_Z(K+Z)$ on the cycle $Z$, we have $H^0(Z, \calL_Z)_{reg} \neq \emptyset$ and it hasn't got base points at intersection points of exceptional divisors. Furthermore let's have a vertex $v \in |l'|_{*}$, so a vertex such that $(E_v, l') < 0$ , then the line bundle $\calL_Z$ hasn't got a base point on the exceptional divisor $E_v$.

3) For an arbitrary vertex $v \in \calv $ let's denote $t_v = (- Z_K+Z, E_v)$, with this notation we have got $\tau( \overline{\im(c^{l'}(Z))})= \prod_{v\in |l'|_{*}}   {t_v \choose (l', E_v)}$.
\end{theorem*}

For an arbitrary singularity the situation is more complicated because although the existence of the cycle $C_{min}(Z, l')$ is ensured by \cite{NNAD} we can't even determine combinatorially for which cycles and Chern classes $Z = C_{min}(Z, l')$ holds, although we prove the inequality part of the previous theorem also in the general case:

\begin{theorem*}
Let $\mathcal{T}$ be an arbitrary resolution graph and $\tX$ a singularity corresponding to it.
Let's have a Chern class $l' \in -S'$ and an integer effective cycle $Z \geq E$, such that $Z = C_{min}(Z, l')$, in particular we know that the Abel map $ \eca^{l'}(Z) \to \im(c^{l'}(Z))$ is birational.

For an arbitrary vertex $v \in \calv $ let's denote $t_v = (- Z_K+Z, E_v)$, with this notation we have got $\tau( \overline{\im(c^{l'}(Z))}) < \prod_{v\in |l'|_{*}}   {t_v \choose (l', E_v)}$.
\end{theorem*}

In section 2) we summarise the necessary background on normal surface singularities.

In section 3) we recall the necessary definitions and results about effective Cartier divisors and Abel maps from \cite{NNI}.

In section 4) we recall our working definition about generic normal surface singularities and the main cohomological results from \cite{NNII}.

In section 5) we recall from \cite{R} the results about relatively generic analytic structures on normal surface singularities.

In section 6) we explain the invariant $\tau$ we investigate in this article and it's relation to the class of the projective clousure and the multiplicity of the infinite hyperplane in the
dual projective variety.

In section 7) we recall the necessary results from \cite{H} about base points of canonical line bundles and hyperelliptic involutions.

In section 8) we recall the structure theorems about images of Abel maps from \cite{NNAD}.

In section 9) we prove our main theorems about the $\tau$ invariant of the varieties $\overline{\im(c^{l'}(Z))}$.

\section{Preliminaries}\label{s:prel}

\subsection{The resolution}\label{ss:notation}
Let $(X,o)$ be the germ of a complex analytic normal surface singularity,
 and let us fix  a good resolution  $\phi:\widetilde{X}\to X$ of $(X,o)$.
We denote the exceptional curve $\phi^{-1}(0)$ by $E$, and let $\{E_v\}_{v\in\calv}$ be
its irreducible components. Set also $E_I:=\sum_{v\in I}E_v$ for any subset $I\subset \calv$.
For the cycle $l=\sum n_vE_v$ let its support be $|l|=\cup_{n_v\not=0}E_v$.
For more details see \cite{NCL,Nfive}.
\subsection{Topological invariants}\label{ss:topol}
Let $\Gamma$ be the dual resolution graph
associated with $\phi$;  it  is a connected graph.
Then $M:=\partial \widetilde{X}$, as a smooth oriented 3--manifold, 
 can be identified with the link of $(X,o)$, it is also
an oriented  plumbed 3--manifold associated with $\Gamma$.
{\it We will assume  (for any singularity we will deal with) that the link
 $M$ is a rational homology sphere,}
or, equivalently,  $\Gamma$ is a tree with all genus
decorations  zero. We use the same
notation $\mathcal{V}$ for the set of vertices.

The lattice $L:=H_2(\widetilde{X},\mathbb{Z})$ is  endowed
with a negative definite intersection form  $I=(\,,\,)$. It is
freely generated by the classes of 2--spheres $\{E_v\}_{v\in\mathcal{V}}$.
 The dual lattice $L':=H^2(\widetilde{X},\mathbb{Z})$ is generated
by the (anti)dual classes $\{E^*_v\}_{v\in\mathcal{V}}$ defined
by $(E^{*}_{v},E_{w})=-\delta_{vw}$, the opposite of the Kronecker symbol.
The intersection form embeds $L$ into $L'$. Then $H_1(M,\mathbb{Z})\simeq L'/L$, abridged by $H$.
Usually one also identifies $L'$ with those rational cycles $l'\in L\otimes \Q$ for which
$(l',L)\in\Z$ (or, $L'={\rm Hom}_\Z(L,\Z)\simeq H^2(\tX,\mathbb{Z})$), where the intersection form extends naturally.

All the $E_v$--coordinates of any $E^*_u$ are strict positive.
We define the Lipman cone as $\calS':=\{l'\in L'\,:\, (l', E_v)\leq 0 \ \mbox{for all $v$}\}$.
It is generated over $\bZ_{\geq 0}$ by $\{E^*_v\}_v$.
We also write $\calS:=\calS'\cap L$.

There is a natural partial ordering of $L'$ and $L$: we write $l_1'\geq l_2'$ if
$l_1'-l_2'=\sum _v r_vE_v$ with all $r_v\geq 0$. We set $L_{\geq 0}=\{l\in L\,:\, l\geq 0\}$ and
$L_{>0}=L_{\geq 0}\setminus \{0\}$.
We will write $Z_{min}\in L$ for the  {\it minimal} (or fundamental, or Artin) cycle, which is
the minimal non--zero cycle of $\calS$ \cite{Artin62,Artin66}.

We define the
  (anti)canonical cycle $Z_K\in L'$ via the {\it adjunction formulae}
$(-Z_K+E_v,E_v)+2=0$ for all $v\in \mathcal{V}$.
(In fact,  $Z_K=-c_1(\Omega^2_{\widetilde{X}})$, cf. (\ref{eq:PIC})).
In a minimal resolution $Z_K\in \calS'$.

Finally we consider the Riemann--Roch expression
 $\chi(l')=-(l',l'-Z_K)/2$ defined for any $l'\in L'$.

\subsection{Some analytic invariants}\label{ss:analinv}
{\bf The group ${\rm Pic}(\widetilde{X})$}
of  isomorphism classes of analytic line bundles on $\widetilde{X}$ appears in the (exponential) exact sequence
\begin{equation}\label{eq:PIC}
0\to {\rm Pic}^0(\widetilde{X})\to {\rm Pic}(\widetilde{X})\stackrel{c_1}
{\longrightarrow} L'\to 0, \end{equation}
where  $c_1$ denotes the first Chern class. Here
$ {\rm Pic}^0(\widetilde{X})=H^1(\widetilde{X},\calO_{\widetilde{X}})\simeq
\C^{p_g}$, where $p_g$ is the {\it geometric genus} of
$(X,o)$. $(X,o)$ is called {\it rational} if $p_g(X,o)=0$.
 Artin in \cite{Artin62,Artin66} characterized rationality topologically
via the graphs; such graphs are called `rational'. By this criterion, $\Gamma$
is rational if and only if $\chi(l)\geq 1$ for any effective non--zero cycle $l\in L_{>0}$.

The epimorphism
$c_1$ admits a unique group homomorphism section $l'\mapsto s(l')\in {\rm Pic}(\widetilde{X})$,
 which extends the natural
section $l\mapsto \calO_{\widetilde{X}}(l)$ valid for integral cycles $l\in L$, and
such that $c_1(s(l'))=l'$  \cite{OkumaRat}.
We call $s(l')$ the  {\it natural line bundles} on $\widetilde{X}$.
By  the very  definition, $\calL$ is natural if and only if some power $\calL^{\otimes n}$
of it has the form $\calO_{\tX}(l)$ for some $l\in L$.

\bekezdes $\mathbf{{Pic}(Z)}.$ \
Similarly, if $Z\in L_{>0}$ is a non--zero effective integral cycle such that its support is $|Z| =E$,
and $\calO_Z^*$ denotes
the sheaf of units of $\calO_Z$, then ${\rm Pic}(Z)=H^1(Z,\calO_Z^*)$ is  the group of isomorphism classes
of invertible sheaves on $Z$. It appears in the exact sequence
  \begin{equation}\label{eq:PICZ}
0\to {\rm Pic}^0(Z)\to {\rm Pic}(Z)\stackrel{c_1}
{\longrightarrow} L'\to 0, \end{equation}
where ${\rm Pic}^0(Z)=H^1(Z,\calO_Z)$.
If $Z_2\geq Z_1$ then there are natural restriction maps,
${\rm Pic}(\widetilde{X})\to {\rm Pic}(Z_2)\to {\rm Pic}(Z_1)$.
Similar restrictions are defined at  ${\rm Pic}^0$ level too.
These restrictions are homomorphisms of the exact sequences  (\ref{eq:PIC}) and (\ref{eq:PICZ}).

Furthermore, we define a section of (\ref{eq:PICZ}) by
$s_Z(l'):=
{\mathcal O}_{\widetilde{X}}(l')|_{Z}$, they also satisfy $c_1\circ s_Z={\rm id}_{L'}$. 
We write  ${\mathcal O}_{Z}(l')$ for $s_Z(l')$, and we call them
 {\it natural line bundles } on $Z$.

We also use the notations ${\rm Pic}^{l'}(\widetilde{X}):=c_1^{-1}(l')
\subset {\rm Pic}(\widetilde{X})$ and
${\rm Pic}^{l'}(Z):=c_1^{-1}(l')\subset{\rm Pic}(Z)$
respectively. Multiplication by $\calO_{\widetilde{X}}(-l')$, or by
$\calO_Z(-l')$, provides natural affine--space isomorphisms
${\rm Pic}^{l'}(\widetilde{X})\to {\rm Pic}^0(\widetilde{X})$ and
${\rm Pic}^{l'}(Z)\to {\rm Pic}^0(Z)$.

\bekezdes\label{bek:restrnlb} {\bf Restricted natural line bundles.}
The following warning is appropriate.
Note that if $\tX_1$ is a connected small convenient  neighbourhood
of the union of some of the exceptional divisors (hence $\tX_1$ also stays as the resolution
of the singularity obtained by contraction of that union of exceptional  curves), then one can repeat the definition of
natural line bundles at the level of $\tX_1$ as well (as a splitting of (\ref{eq:PIC}) applied for
$\tX_1$). However, the restriction to
$\tX_1$ of a natural line bundle of $\tX$ (even of type
$\calO_{\tX}(l)$ with $l$ integral cycle supported on $E$)  is usually not natural on $\tX_1$:
$\calO_{\tX}(l')|_{\tX_1}\not= \calO_{\tX_1}(R(l'))$
 (where $R:H^2(\tX,\Z)\to H^2(\tX_1,\Z)$ is the natural cohomological 
 restriction), though their Chern classes coincide.

Therefore, in inductive procedure when such restriction is needed,
 we will deal with the family of {\it restricted natural line bundles}. This means the following.
If we have two resolution spaces $\tX_1 \subset \tX$ with resolution graphs $\mathcal{T}_1 \subset \mathcal{T}$ and we have a Chern class $l' \in L'$, then we denote 
by $\calO_{\tX_1}(l') = \calO_{\tX}(l') | \tX_1$ the restriction of the natural line bundle $\calO_{\tX}(l')$.
Similarly if $Z$ is an effective integer cycle on $\tX$ with maybe $|Z| \neq E$, then we denote $\calO_{Z}(l') = \calO_{\tX}(l') | Z$.

Furthermore if $\calL$ is a line bundle on $\tX_1$, then we denote $\calL(l') = \calL \otimes \calO_{\tX}(l')$.
Similarly if $Z$ is  an effective integer cycle on $\tX$ and $\calL$ is a line bundle on $Z$, then we denote $\calL(l') = \calL \otimes \calO_Z(l')$.

\bekezdes \label{bek:ansemgr} {\bf The analytic semigroups.} \
By definition, the analytic semigroup associated with the resolution $\tX$ is

\begin{equation}\label{eq:ansemgr}
\calS'_{an}:= \{l'\in L' \,:\,\calO_{\tX}(-l')\ \mbox{has no  fixed components}\}.
\end{equation}
It is a subsemigroup of $\calS'$. One also sets $\calS_{an}:=\calS_{an}'\cap L$, a subsemigroup
of $\calS$. In fact, $\calS_{an}$
consists of the restrictions   ${\rm div}_E(f)$ of the divisors
${\rm div}(f\circ \phi)$ to $E$, where $f$ runs over $\calO_{X,o}$. Therefore, if $s_1, s_2\in \calS_{an}$, then
${\rm min}\{s_1,s_2\}\in \calS_{an}$ as well (take the generic linear combination of the corresponding functions).
In particular,  for any $l\in L$, there exists a {\it unique} minimal
$s\in \calS_{an}$ with $s\geq l$.

Similarly, for any $h\in H=L'/L$ set $\calS'_{an,h}:\{l'\in \calS_{an}\,:\, [l']=h\}$.
Then for any  $s'_1, s'_2\in \calS_{an,h}$ one has
${\rm min}\{s'_1,s'_2\}\in \calS_{an,h}$, and
for any $l'\in L'$   there exists a unique minimal
$s'\in \calS_{an,[l']}$ with $s'\geq l'$.

For any $l'\in\calS_{an}'$ there exists an ideal sheaf $\cali(l')$ with 0--dimensional support along $E$ such that
 $H^0(\tX,\calO_{\tX}(-l'))\cdot \calO_{\widetilde{X}}=\calO_{\widetilde{X}}(-l')\cdot \cali(l')$.
The ideal $\cali(l') $ describes the space of base points of the line bundle $\calO_{\tX}(-l')$.

If $l'\in\calS'_{an}$ and  the divisor of a generic global section of $\calO_{\tX}(-l')$ intersects
$E_v$, then $(l',E_v)<0$. In particular, if $p\in E_v$ is a  base point then necessarily $(l',E_v)<0$.

Choose a base point $p$ of $\calO_{\tX}(-l')$, and assume that it is a regular point of $E$, and that $\cali(l')_p$
 in the
local ring $\calO_{\tX,p}$ is $(x^t,y)$, where $x,y$ are some local coordinates at $p$  with $\{x=0\}=E$ (locally),
and $t\geq 1$.
Then we say that $p$ is a {\it $t$--simple base point}. In such cases we write
$t=t(p)$. Furthermore, $p$ is called {\it simple}
if it is $t$--simple for some $t\geq 1$.

Let's have a Chern class $l' \in S'_{an}$ and let's have a base point $p \in E_{v, reg}$ of a natural line $\calO_{\tX}(-l')$, which is simple, there is another interpretation of the positive integer $t$, such that $p$ is $t$-simple.

Let's have a generic section in $s \in H^0(\calO_{\tX}(-l'))$ and $D = |s|$, then we know, that $D$ has a cut $D'$, which is transversal at the base point $p$.

Let's blow up the exceptional divisor $E_v$ along the cut $D'$ sequentially, so let's blow up first at the point $p$ and let the new exceptional divisor be $E_{v_1}$ and let's denote
the strict transform of the cut $D'$ with the same notation.
Then let's blow up $E_{v_1}$ at the intersection point $E_{v_1} \cap D'$ and let the new exceptional divisor be $E_{v_2}$ and so on.

Let's denote the given resolution at the $i$-th step by $\tX_i$ with the blow up map $b_i : \tX_i \to \tX$ and let's look at the natural line bundle $\calL_i = \calO_{\tX_i}(-b_i^*(l') - \sum_{1 \leq j \leq i} j \cdot E_{v_j}) = \calO_{\tX_i}(D_{st})$, where $D_{st}$ is the strict transform of the divisor $D$.

Let $t$ be the minimal number, such that $\calL_t$ hasn't got a base point along the excpetional divisor $E_{v_t}$.
Equivalently $t$ is the maximal integer, such that $H^0(\tX_t, \calL_t) = H^0(\calO_{\tX_t}( - b_t^*(l')))$ and $h^1(\tX_t, \calL_t)  = h^1(\calO_{\tX}(-l')) + t$. 

In this case $p$ is a $t$-simple base point ot the natural line bundle $\calO_{\tX}(-l')$.

\section{Effective Cartier divisors and Abel maps}

  In this section we review some needed material from \cite{NNI}.

We fix a good resolution $\phi:\tX\to X$ of a normal surface singularity,
whose link is a rational homology sphere. 

\subsection{} \label{ss:4.1}
Let us fix an effective integral cycle  $Z\in L$, $Z\geq E$. (The restriction $Z\geq E$ is imposed by the
easement of the presentation, everything can be adopted  for $Z>0$).

Let $\eca(Z)$  be the space of effective Cartier (zero dimensional) divisors supported on  $Z$.
Taking the class of a Cartier divisor provides  a map
$c:\eca(Z)\to \pic(Z)$.
Let  $\eca^{l'}(Z)$ be the set of effective Cartier divisors with
Chern class $l'\in L'$, that is,
$\eca^{l'}(Z):=c^{-1}(\pic^{l'}(Z))$.

We consider the restriction of $c$, $c^{l'}:\eca^{l'}(Z)
\to \pic^{l'}(Z)$ too, sometimes still denoted by $c$. 

For any $Z_2\geq Z_1>0$ one has the natural  commutative diagram
\begin{equation}\label{eq:diagr}
\begin{picture}(200,45)(0,0)
\put(50,37){\makebox(0,0)[l]{$
\eca^{l'}(Z_2)\,\longrightarrow \, \pic^{l'}(Z_2)$}}
\put(50,8){\makebox(0,0)[l]{$
\eca^{l'}(Z_1)\,\longrightarrow \, \pic^{l'}(Z_1)$}}
\put(70,22){\makebox(0,0){$\downarrow$}}
\put(135,22){\makebox(0,0){$\downarrow$}}
\end{picture}
\end{equation}

As usual, we say that $\calL\in \pic^{l'}(Z)$ has no fixed components if
\begin{equation}\label{eq:H_0}
H^0(Z,\calL)_{reg}:=H^0(Z,\calL)\setminus \bigcup_v H^0(Z-E_v, \calL(-E_v))
\end{equation}
is non--empty.
Note that $H^0(Z,\calL)$ is a module over the algebra
$H^0(\calO_Z)$, hence one has a natural action of $H^0(\calO_Z^*)$ on
$H^0(Z, \calL)_{reg}$. This second action is algebraic and free.  Furthermore,
 $\calL\in \pic^{l'}(Z)$ is in the image of $c$ if and only if
$H^0(Z,\calL)_{reg}\not=\emptyset$. In this case, $c^{-1}(\calL)=H^0(Z,\calL)_{reg}/H^0(\calO_Z^*)$.

One verifies that $\eca^{l'}(Z)\not=\emptyset$ if and only if $-l'\in \calS'\setminus \{0\}$. Therefore, it is convenient to modify the definition of $\eca$ in the case $l'=0$: we (re)define $\eca^0(Z)=\{\emptyset\}$,
as the one--element set consisting of the `empty divisor'. We also take $c^0(\emptyset):=\calO_Z$, then we have
\begin{equation}\label{eq:empty}
\eca^{l'}(Z)\not =\emptyset \ \ \Leftrightarrow \ \ l'\in -\calS'.
\end{equation}
If $l'\in -\calS'$  then
  $\eca^{l'}(Z)$ is a smooth variety of dimension $(l',Z)$. Moreover,
if $\calL\in \im (c^{l'}(Z))$ (the image of $c^{l'}$)
then  the fiber $c^{-1}(\calL)$
 is a smooth, irreducible quasiprojective variety of  dimension
 \begin{equation}\label{eq:dimfiber}
\dim(c^{-1}(\calL))= h^0(Z,\calL)-h^0(\calO_Z)=
 (l',Z)+h^1(Z,\calL)-h^1(\calO_Z).
 \end{equation}

\bekezdes \label{bek:I}
Consider again  a Chern class $l'\in-\calS'$ as above.
The $E^*$--support $I(l')\subset \calv$ of $l'$ is defined via the identity  $l'=\sum_{v\in I(l')}a_vE^*_v$ with all
$\{a_v\}_{v\in I}$ nonzero. Its role is the following.

Besides the Abel map $c^{l'}(Z)$ one can consider its `multiples' $\{c^{nl'}(Z)\}_{n\geq 1}$ as well. It turns out
(cf. \cite[\S 6]{NNI})   that $n\mapsto \dim \im (c^{nl'}(Z))$
is a non-decreasing sequence, and   $\im (c^{nl'}(Z))$ is an affine subspace
for $n\gg 1$, whose dimension $e_Z(l')$ is independent of $n\gg 0$, and essentially it depends only
on $I(l')$.
We denote the linearisation of this affine subspace by $V_Z(I) \subset H^1(\calO_Z)$ or if the cycle $Z \gg 0$, then $ V_{\tX}(I) \subset H^1(\calO_{\tX})$.

Moreover, by \cite[Theorem 6.1.9]{NNI},
\begin{equation}\label{eq:ezl}
e_Z(l')=h^1(\calO_Z)-h^1(\calO_{Z|_{\calv\setminus I(l')}}),
\end{equation}
where $Z|_{\calv\setminus I(l')}$ is the restriction of the cycle $Z$ to its $\{E_v\}_{v\in \calv\setminus I(l')}$
coordinates.

If $Z\gg 0$ (i.e. all its $E_v$--coordinated are very large), then (\ref{eq:ezl}) reads as
\begin{equation}\label{eq:ezlb}
e_Z(l')=h^1(\calO_{\tX})-h^1(\calO_{\tX(\calv\setminus I(l'))}),
\end{equation}
where $\tX(\calv\setminus I(l'))$ is a convenient small neighbourhood of $\cup_{v\in \calv\setminus I(l')}E_v$.

Let $\Omega _{\tX}(I)$ be the subspace of $H^0(\tX\setminus E, \Omega^2_{\tX})/ H^0(\tX,\Omega_{\tX}^2)$ generated by differential forms which have no poles along $E_I\setminus \cup_{v\not\in I}E_v$.
Then, cf. \cite[\S8]{NNI},
\begin{equation}\label{eq:ezlc}
h^1(\calO_{\tX(\calv\setminus I)})=\dim \Omega_{\tX}(I).
\end{equation}

Similarly let $\Omega _{Z}(I)$ be the subspace of $H^0(\calO_{\tX}(K + Z))/ H^0(\calO_{\tX}(K))$ generated by differential forms which have no poles along $E_I\setminus \cup_{v\not\in I}E_v$.
Then, cf. \cite[\S8]{NNI},
\begin{equation}\label{eq:ezlc}
h^1(\calO_{Z_{(\calv\setminus I)}})=\dim \Omega_{Z}(I).
\end{equation}

We have also the following duality from \cite{NNI} supporting the equalities above:

\begin{theorem}\cite{NNI}\label{th:DUALVO}
Via Laufer duality one has  $V_{\tX}(I)^*=\Omega_{\tX}(I)$ and $V_{Z}(I)^*=\Omega_Z(I)$.
\end{theorem}

\section{Analytic invariants of generic analytic type}\label{s:AnGen}

For a precise working definition of a generic analytic type see \cite{NNII}, \cite{NNM}, \cite{R}, in a slightly simplified language we can regard  the generic analytic structure in the following way as well.
Fix a graph $\Gamma$. For each $E_v$ ($v\in\calv$) the disc bundle with Euler number $E_v^2$ is taut:
it has no analytic moduli. The generic $\tX$ is obtained by gluing `generically' these bundles according to the edges of $\Gamma$ (as an analytic plumbing).

\subsection{Review of some results of \cite{NNII}}\label{ss:ReviewNNII}

 The list of analytic invariants, associated with a generic analytic type
  (with respect to a fixed resolution graph),
 which in \cite{NNII} are described topologically include  the following ones:
 $h^1(\calO_Z)$, $h^1(\calO_Z(l'))$ (with certain restriction on the Chern class $l'$),
  --- this last one applied  for $Z\gg 0$ provides  $h^1(\calO_{\tX})$
 and  $h^1(\calO_{\tX}(l')$) too ---,
the analytic semigroup, 
and the maximal ideal cycle of $\tX$.
See above or \cite{CDGPs,CDGEq,Lipman,Nfive,NPS,NCL,Ok,MR}
for the definitions and relationships between them.
The topological characterizations use the Riemann--Roch expression $\chi:L'\to \bQ$.

In the next theorem  the bundles $\calO_{\tX}(-l')$ are the `restricted natural line bundles'
associated with some pair $\tX\subset \tX_{top}$. In particular, it is valid even if $\tX_{top}=\tX$ and the bundles are natural line bundles.
The theorem (and basically several statements
 regarding generic analytic structure and restricted natural line bundles) says that these bundles behave cohomologically
as the generic line bundles in ${\rm Pic}^{-l'}(\tX)$ (for more comments  see \cite{NNII}).

\begin{theorem}\cite[Theorem A]{NNII}\label{th:OLD}
 Fix a resolution graph $\mathcal{T}$ (tree of $\bP^1$'s) and let's have a generic analytic type $\tX$ corresponding to it. Then
the following identities hold:\\
(a) For any effective cycle $Z\in L_{>0}$, such that the support $|Z|$ is connected, we have
\begin{equation*}
h^1(\calO_Z) = 1-\min_{0< l \leq Z,l\in L}\{\chi(l)\}.
\end{equation*}
(b) If $l'=\sum_{v\in \calv}l'_vE_v \in L'$ satisfies
$l'_v >0$ for any $E_v$ in the support of $Z$ then
\begin{equation*}
h^1(Z,\calO_Z(-l'))=\chi(l')-\min _{0\leq l\leq Z, l\in L}\, \{\chi(l'+l)\}.
\end{equation*}
(c) If $p_g(X,o)=h^1(\tX,\calO_{\tX})$ is the geometric genus of $(X,o)$ then
\begin{equation*}
p_g(X,o)= 1-\min_{l\in L_{>0}}\{\chi(l)\} =-\min_{l\in L}\{\chi(l)\}+\begin{cases}
1 & \mbox{if $(X,o)$ is not rational}, \\
0 & \mbox{else}.
\end{cases}
\end{equation*}
(d) More generally, for any $l'\in L'$
\begin{equation*}
h^1(\tX,\calO_{\tX}(-l'))=\chi(l')-\min _{l\in L_{\geq 0}}\, \{\chi(l'+l)\}+
\begin{cases}
1 & \mbox{if \ $l'\in L_{\leq 0}$ and $(X,o)$ is not rational}, \\
0 & \mbox{else}.
\end{cases}
\end{equation*}
(e) For $l\in L$ set $\mathfrak{h}(l)=\dim ( H^0(\tX, \calO_{\tX})/ H^0(\tX, \calO_{\tX}(-l)))$.
Then $\mathfrak{h}(0)=0$ and for $l_0>0$ one has
\begin{equation*}
\mathfrak{h}(l_0)=\min_{l\in L_{\geq 0}} \{\chi(l_0+l)\}-\min_{l\in L_{\geq 0}} \{\chi(l)\}+
\begin{cases}
1 & \mbox{if $(X,o)$ is not rational}, \\
0 & \mbox{else}.
\end{cases}
\end{equation*}
(f) \ 
$\calS'_{an}= \{l'\,:\, \chi(l')<
\chi(l' +l) \ \mbox{for any $l\in L_{>0}$}\}\cup\{0\}$.\\
(g)
 Assume that $\Gamma$ is a non--rational graph  and set
$\calm=\{ Z\in L_{>0}\,:\, \chi(Z)=\min _{l\in L}\chi(l)\}$.
Then 
the unique maximal element of $\calm$ is the maximal ideal cycle of\, $\tX$.

(Note that in the above formulae one also has $\min _{l\in L_{\geq 0}}\{\chi(l)\}=\min _{l\in L}\{\chi(l)\}$.)
\end{theorem}

\section{The relative setup.}

In this section we wish to summarise the results from \cite{R} about relatively generic analytic structures we need in this article. 

We consider an effective integer cycle $Z$ on a resolution $\tX$ with resolution graph $\mathcal{T}$, and a smaller cycle $Z_1 \leq Z$, where we denote $|Z_1| = \calv_1$ and the subgraph corresponding to it by $\mathcal{T}_1$.

We have the restriction map $r: \pic(Z)\to \pic(Z_1)$ and one has also the (cohomological) restriction operator
  $R_1 : L'(\mathcal{T}) \to L_1':=L'(\mathcal{T}_1)$
(defined as $R_1(E^*_v(\mathcal{T}))=E^*_v(\mathcal{T}_1)$ if $v\in \calv_1$, and
$R_1(E^*_v(\mathcal{T}))=0$ otherwise).

For any $\calL\in \pic(Z)$ and any $l'\in L'(\mathcal{T})$ it satisfies
\begin{equation*}
c_1(r(\calL))=R_1(c_1(\calL)).
\end{equation*}

In particular, we have the following commutative diagram as well:

\begin{equation*}  
\begin{picture}(200,40)(30,0)
\put(50,37){\makebox(0,0)[l]{$
\ \ \eca^{l'}(Z)\ \ \ \ \ \stackrel{c^{l'}(Z)}{\longrightarrow} \ \ \ \pic^{l'}(Z)$}}
\put(50,8){\makebox(0,0)[l]{$
\eca^{R_1(l')}(Z_1)\ \ \stackrel{c^{R_1(l')}(Z_1)}{\longrightarrow} \  \pic^{R_1(l')}(Z_1)$}}
\put(162,22){\makebox(0,0){$\downarrow \, $\tiny{$r$}}}
\put(78,22){\makebox(0,0){$\downarrow \, $\tiny{$\fr$}}}
\end{picture}
\end{equation*}

By the `relative case' we mean that instead of the `total' Abel map
$c^{l'}(Z)$ we study its restriction above a fixed fiber of $r$.

That is, we fix some  $\mfl\in \pic^{R_1(l')}(Z_1)$, and we study
the restriction of $c^{l'}(Z)$ to $(r\circ c^{l'}(Z))^{-1}(\mfl)\to r^{-1}(\mfl)$.

 The subvariety $(r\circ c^{l'}(Z))^{-1}(\mfl)
=(c^{R_1(l')}(Z_1) \circ \fr)^{-1}(\mfl) \subset \eca^{l'}(Z)$ is denoted by $\eca^{l', \mfl}(Z)$.

\begin{theorem}\cite{R}
Fix an arbitrary singularity $\tX$ a Chern class $l'\in -\calS'$, an integer effective cycle $Z\geq E$ and a subcycle $Z_1 \leq Z$ and let's have a line bundle  $\mfl\in \pic^{R(l')}(Z_1)$.
Assume that  $\eca^{l', \mfl}(Z)$ is nonempty, then it is smooth of dimension $h^1(Z_1,\mfl)  - h^1(\calO_{Z_1})+ (l', Z)$ and irreducible.
\end{theorem}

Let's recall from \cite{R} the analouge of the theroems about dominance of Abel maps in the relative setup:

\begin{definition}\cite{R}
Fix an arbitrary singularity $\tX$, a Chern class $l'\in -\calS'$, an integer effective cycle $Z\geq E$, a subcycle $Z_1 \leq Z$ and a line bundle $\mfl\in \pic^{R_1(l')}(Z_1)$ as above.
We say that the pair $(l',\mfl ) $ is {\it relative
dominant} on the cycle $Z$, if the closure of $ r^{-1}(\mfl)\cap \im(c^{l'}(Z))$ is $r^{-1}(\mfl)$.
\end{definition}

\begin{theorem}\label{th:dominantrel}\cite{R}
 One has the following facts:

(1) If $(l',\mfl)$ is relative dominant on the cycle $Z$, then $ \eca^{l', \mfl}(Z)$ is
nonempty and $h^1(Z,\calL)= h^1(Z_1,\mfl)$ for any
generic line bundle $\calL\in r^{-1}(\mfl)$.

(2) $(l',\mfl)$ is relative dominant on the cycle $Z$,  if and only if for all
 $0<l\leq Z$, $l\in L$ one has
$$\chi(-l')- h^1(Z_1, \mfl) < \chi(-l'+l)-
 h^1((Z-l)_1, \mfl(-l)).$$, where we denote $(Z-l)_1 = \min(Z-l, Z_1)$.
\end{theorem}

\begin{theorem}\label{th:hegy2rel}\cite{R}
Fix an arbitrary singularity $\tX$, a Chern class $l'\in -\calS'$, an integer effective cycle $Z\geq E$, a subcycle $Z_1 \leq Z$ and a line bundle $\mfl\in \pic^{R_1(l')}(Z_1)$ as in Theorem \ref{th:dominantrel}. 
Then for any $\calL\in r^{-1}(\mfl)$ one has
\begin{equation*}\label{eq:genericLrel}
\begin{array}{ll}h^1(Z,\calL)\geq \chi(-l')-
\min_{0\leq l\leq Z,\ l\in L} \{\,\chi(-l'+l) -
h^1((Z-l)_1, \mfl(-l))\, \}, \ \ \mbox{or equivalently,}\\
h^0(Z,\calL)\geq \max_{0\leq l\leq Z,\, l\in L}
\{\,\chi(Z-l,\calL(-l))+  h^1((Z-l)_1, \mfl(-l))\,\}.\end{array}\end{equation*}
Furthermore, if $\calL$ is generic in $r^{-1}(\mfl)$
then in both inequalities we have equalities.
\end{theorem}

In the following we recall the results from \cite{R} about relatively generic analytic structures:

Let's fix a a topological type, in other words a resolution graph $\mathcal{T}$ with vertex set $\calv$,
we consider a partition $\calv = \calv_1 \cup  \calv_2$ of the set of vertices $\calv=\calv(\mathcal{T})$.

They define two (not necessarily connected) subgraphs $\mathcal{T}_1$ and $\mathcal{T}_2$.

We call the intersection of an exceptional divisor from
$ \calv_1 $ with an exceptional divisor from  $ \calv_2 $ a
{\it contact point}.
 For any $Z\in L=L(\mathcal{T})$ we write $Z=Z_1+Z_2$,
where $Z_i\in L(\mathcal{T}_i)$ is
supported in $\mathcal{T}_i$ ($i=1,2$).
Furthermore, parallel to the restrictions
$r_i : \pic(Z)\to \pic(Z_i)$ one also has the (cohomological) restriction operators
  $R_i : L'(\mathcal{T}) \to L_i':=L'(\mathcal{T}_i)$
(defined as $R_i(E^*_v(\mathcal{T}))=E^*_v(\mathcal{T}_i)$ if $v\in \calv_i$, and
$R_i(E^*_v(\mathcal{T}))=0$ otherwise).
For any $l'\in L'(\mathcal{T})$ and any $\calL\in \pic^{l'}(Z)$ it satisfies $c_1(r_i(\calL))=R_i(c_1(\calL))$.

In the following for the sake of simplicity we will denote $r = r_1$ and $R = R_1$.

Furthermore let's have a fixed analytic type $\tX_1$ for the subtree $\mathcal{T}_1$ (if it is disconnected, then an analytic type for each connected component).

Also for each vertex $v_2 \in \calv_2$ which has got a neighbour $v_1$ in $\calv_1$ we fix a cut $D_{v_2}$ on $\tX_1$, along we glue the exceptional divisor $E_{v_2}$.
This means that $D_{v_2}$ is a divisor, which intersects the exceptional divisor $E_{v_1}$ transversally in one point and we will glue the exceptional divisor $E_{v_2}$ in a way, 
such that $E_{v_2} \cap \tX_1$ equals $D_{v_2}$.

If for some vertex $v_2 \in \calv_2$, which has got a neighbour in $\calv_1$ we don't say explicitely what is the fixed cut, then it should be understood in the way that we glue the exceptional divisor $E_{v_2}$ along a generic cut.

Let's plumb the tubular neihgbourhoods of the exceptional divisors $E_{v_2}, v_2 \in \calv_2$  with the above conditions generically to the fixed resolution $\tX_1$, we get a singularity $\tX$ with resolution graph $\mathcal{T}$ and we say that $\tX$ is a relatively generic singularity corresponding to the analytical structure $\tX_1$ and the cuts $D_{v_2}$, for the more precise explanation of genericity look at \cite{R}.

We have the following theorem with this setup from \cite{R}:

\begin{theorem}\cite{R}\label{relgen}
Let's have the setup as above, so two resolution graphs $\mathcal{T}_1 \subset \mathcal{T} $ with vertex sets $\calv_1 \subset \calv$, where $\calv = \calv_1 \cup \calv_2$  and a fixed singularity $\tX_1$ for the resolution graph $\mathcal{T}_1$, and cuts $D_{v_2}$ along we glue $E_{v_2}$ for all vertices $v_2 \in \calv_2$, which have got a neighbour in $\calv_1$.

Assume that $\tX$ has a relatively generic analytic stucture on $\mathcal{T}$ corresponding to $\tX_1$ and the cuts $D_{v_2}$.

Furthermore let's have an effective cycle $Z$ on $\tX$ and let's have $Z = Z_1 + Z_2$, where $|Z_1| \subset \calv_1$ and $|Z_2| \subset \calv_2$.

1) 
Let's have the natural line bundle $\calL= \calO_{\tX}(l')$ on $\tX$, such that $ l' = - \sum_{v \in \calv} a_v E_v$, with $a_v > 0, v \in \calv_2 \cap |Z|$, and let's denote $c_1 (\calL | Z) = l'_ m \in L'_{|Z|}$, furthermore let's denote $\mfl = \calL | Z_1$, then we have the following:

We have $H^0(Z,\calL)_{reg} \not=\emptyset$ if and only if $(l',\mfl)$ is relative dominant on the cycle $Z$ or equivalently:

\begin{equation*}
\chi(-l')- h^1(Z_1, \mfl) < \chi(-l'+l)-  h^1((Z-l)_1, \mfl(-l)),
\end{equation*}
for all $0 < l \leq Z$.

2) 

Let's have the same setup as in part 1), then we have:

\begin{equation*}
h^1(Z, \calL) =  h^1(Z, \calL_{gen}),                                   
\end{equation*}
where $\calL_{gen}$ is a generic line bundle in $ r^{-1}(\mfl) \subset \pic^{l'_m}(Z)$, or equivalently:

\begin{equation*}
h^1(Z, \calL)= \chi(-l') - \min_{0 \leq l \leq Z}(\chi(-l'+l)-  h^1((Z-l)_1, \mfl(-l))).
\end{equation*}

3) 

Let's have the natural line bundle $\calL= \calO_{\tX}(l')$ on $\tX$, such that $ l' = - \sum_{v \in \calv} a_v E_v$, and assume that  $a_v \neq 0$ if $ v \in \calv_2 \cap |Z|$.
Let's denote $c_1 (\calL | Z) = l'_ m \in  L'_{|Z|}$ and $\mfl = \calL | Z_1$.

Assume that $H^0(Z,\calL)_{reg} \not=\emptyset$, and pick an arbitrary $D \in (c^{l'_m}(Z))^{-1}\calL \subset \eca^{l'_m, \mfl}(Z)$.
Then $ c^{l'_m}(Z) : \eca^{l'_m, \mfl}(Z) \to r^ {-1}(\mfl  )$ is a submersion in $D$, and $h^1(Z,\calL) = h^1(Z_1, \mfl)$.

In particular the map $ c^{l'_m}(Z) :  \eca^{l'_m, \mfl}(Z) \to r^ {-1}(\mfl  )$ is dominant, which means that $(l'_m,\mfl)$ is relative dominant on the cycle $Z$, or equivalently:

\begin{equation*}
\chi(-l')- h^1(Z_1, \mfl) < \chi(-l'+l)-  h^1((Z-l)_1, \mfl(-l)),
\end{equation*}
for all $0 < l \leq Z$.
\end{theorem}

\begin{remark}

In the theorem above in any formula one can replace $l'$ with $l'_m$, since for every $0 \leq l \leq Z$ one has $\chi(-l')- \chi(-l'+l) = \chi(-l'_m)- \chi(-l'_m+l) = -(l', l) - \chi(l)$.
\end{remark}

\section{Class of irreducible affine varieties}

Let's have a complex vector space $V$ and an irreducible affine subvariety $X \subset V$. 

Let's pick a generic element $w \in V^*$, and let's denote the number of smooth points $p$ of $X$, such that $w$ vanishes on $T_p(X)$ by $\tau(X)$.

Let's denote the projective closure of the affine variety $X$ in the projective closure of $V$ by $\overline{X}$, this is a projective subvariety, and let's denote it's dual projective variety by $(\overline{X})^{*}$.

One can see easily that $\tau(X) > 0$ if and only if $(\overline{X})^{*}$ is a hypersurface in the dual projective space and $\tau(X) = 0$ otherwise let's assume that $\tau(X) > 0$ in the following.

The degree of the projective variety $(\overline{X})^{*}$ is the class of the projective variety $\overline{X}$, let's denote it by $cl(\overline{X})$.

The number $cl(\overline{X})$ is closely related to $\tau(X)$ as explained in the following:

Let $\mu$ be the multiplicity of the infinite hyperplane in the dual variety $(\overline{X})^{*}$ (if it doesn't contain the infinite hyperplane, then $\mu = 0$), we claim that $cl(\overline{X}) = \mu + \tau(X)$.

Indeed the affine hyperplanes, on which $w \in V^*$ vanishes and the infinite hyperplane form a generic pencil of projective hyperplanes through the infinite hyperplane, and $\tau(X)$ is exactly the number of intersection points of this pencil with $(\overline{X})^{*}$, where we don't count the multiplicity of intersection at the infinite hyperplane, but this is exactly $\mu$ and this yields the statement.

We have the following easy florklore lemma about the behaviour of the invariant $\tau$ with respect to direct sums:
\begin{lemma}\label{dir}
Let $V_i$ be complex vector spaces for $1 \leq i \leq n$ and $X_i \subset V_i$ be irreducible subvarieties.
Let's denote $V =  \oplus_{1 \leq i \leq n} V_i$ and $X = \oplus_{1 \leq i \leq n} X_i$, then we have $\tau(X) = \prod_{1 \leq i \leq n} \tau(X_i)$.
\end{lemma}
\begin{proof}
Notice that we have $V^* = \oplus_{1 \leq i \leq n} V_i^*$.
Now let's have a generic element $w \in V^* $, this gives generic elements $w_i \in  V_i^*$.
We know that there are $\tau(X_i)$ number of smooth points $p_{i, 1}, \cdots, p_{i, \tau(X_i)}$ of $X_i$, such that $w_i$ vanishes on $T_{p_{i, j}}X_i$, where $1 \leq j \leq \tau(X_i)$.
Now the points $p$ in $X$, such that $w$ vanishes on $T_p X$ are exactly the points $p = (p_{1, j_1}, \cdots, p_{n, j_n})$, where $1 \leq j_i \leq \tau(X_i)$ for all $1 \leq i \leq n$, this proves the statement.
\end{proof}

\section{Base points of canonical line bundles and hyperelliptic involutions on generic singularities}

In the following we recall some nessacary material from \cite{H} about generic analytic types which will be crucial in the proof of our main theorem.

In the following we recall a few lemmas about the base points of the line bundle $\calO_Z(K + Z)$, where $Z$ is an integer effective cycle on a generic singularity and 
$H^0(\calO_Z( K+Z))_{reg} \neq \emptyset$:

\begin{lemma}\label{baseI}\cite{H}
Assume that $\mathcal{T}$ is a resolution graph and $\tX$ is a generic singularity corresponding to it, and let's have a cycle $Z$ on it, such that $|Z| = \calv$, and $H^0(\calO_{Z}( Z+K))_{reg} \neq \emptyset$.
The line bundle $\calL_{Z} = \calO_Z(K + Z)$ hasn't got a basepoint at the intersection points of exceptional divisors.
\end{lemma}

\begin{lemma}\cite{H}\label{baseII}
Assume that $\mathcal{T}$ is an arbitrary resolution graph and $\tX$ a generic singularity corresponding to it, and let's have a cycle $Z$ on it, such that $|Z| = \calv$, and $H^0(\calO_Z(Z+K))_{reg} \neq \emptyset$.
Assume that $v \in |Z|$ and $Z_v = 1$, then the line bundle $\calL_{Z} = \calO_Z(K + Z)$ hasn't got a basepoint on the regular part of the exceptional divisor $E_v$.
\end{lemma}

In the following let's recall the two main theorems from \cite{H} about hyperelliptic involutions on generic normal surface singularities:

We consider an integer effective cycle $Z$ on the resolution $\tX$ and investigate the existence of a complete linear series $g_2^1$ on it, we have the following two main theorems:

\begin{theorem*}\cite{H}\label{hyper1}
Let's have an arbitrary resolution graph $\mathcal{T}$ and a generic resolution $\tX$ corresponding to it, and let's have an effective integer cycle $Z \geq E$ such that 
$H^0(\calO_Z(K+Z))_{reg} \neq \emptyset$ and two vertices $u', u''$, such that $Z_{u'} = Z_{u''} = 1$ and assume that $e_Z(u', u'') \geq 3$.

With these conditions for every line bundle $\calL \in \im(c^{-E_{u'}^* - E_{u''}^*}(Z))$ one has $h^0(Z, \calL) = 1$.
\end{theorem*}

\begin{theorem*}\cite{H}\label{hyper2}
Let's have an arbitrary resolution graph $\mathcal{T}$ and a generic resolution $\tX$ corresponding to it, and let's have an effective integer cycle $Z \geq E$, such that 
$H^0(\calO_Z(K+Z))_{reg} \neq \emptyset$ and a vertex $u \in \calv$, such that $Z_{u} = 1$.

Assume furthemore that $e_Z(u) \geq 3$, with these conditions for every line bundle $\calL \in \im(c^{-2E_{u}^*}(Z))$ one has $h^0(Z, \calL) = 1$.
\end{theorem*}

\section{Differential forms and fibration theorem}

In this subsection we recall some nessecary background from \cite{NNAD} and we use it to reduce the investigation of the invariant $\tau$ in the case of 
images of Abel maps of normal surface singularities to some special cases.

Let's have a normal surface singularity and an effective integer cycle $Z$ on it, and a Chern class $l' \in -S'$, in the following we denote $d_{Z, l'} = \dim(\im(c^{l'}(Z))$.

Let's recall the following  theorem about the numbers $d_{Z}(l')$ from \cite{NNAD}:

\begin{theorem}\cite{NNAD}\label{imdim}
For $l' \in - S'$ and $Z \geq E$ one has:
\begin{equation}
d_{Z}(l') = \min_{0 \leq Z_1 \leq Z} ((l', Z_1) + h^1(\calO_Z) - h^1(\calO_{Z_1})).
\end{equation}
\end{theorem}

Let's recall also the following results from \cite{NNAD}:

\begin{lemma}\label{lem:MINSETS}\cite{NNAD}
The following three sets of cycles coincide (for fixed $Z\geq E$ and $l'\in-\calS'$ as above):

(I) the set of cycles $Z_1$ with $0\leq Z_1\leq Z$ realizing the minimality in Theorem\ref{NNAD}, that is: $d_Z(l')=(l',Z_1)+h^1(\calO_Z)-h^1(\calO_{Z_1})$.

(II) the set of cycles $Z_1$ with $0\leq Z_1\leq Z$ such that
 {\it (i)} the map
$\eca^{l'}(Z) \to H^1(Z_1)$ is birational onto its image,
and {\it (ii) }the generic fibres of the restriction of $r$,
$r^{im}:\im(c^{l'}(Z)) \to \im(c^{l'}(Z_1))$,  have dimension $h^1(\calO_{Z})-h^1(\calO_{Z_1})$.
(That is, the fibers of $r^{im}$   have maximal possible dimension.)

(III)  the set of cycles $Z_1$ with $0\leq Z_1\leq Z$ such that  for the generic element
$\calL_{gen}^{im}\in\im (c^{l'}(Z))$  and arbitrary section $s\in H^0(Z_1, \calL_{gen}^{im})_{reg}$ with divisor $D$
{\it (i)} in the (analogue of the
Mittag-Lefler sequence associated with the exact sequence $0\to \calO_{Z_1}\stackrel{\times s}{\longrightarrow} \calL_{gen}^{im} \to \calO_D\to 0$, cf. \cite[3.2]{NNI}),
$$0\to H^0(\calO_{Z_1})\stackrel{\times s}{\longrightarrow} H^0(Z_1,\calL_{gen}^{im})
\to \bC^{(Z_1,l')}
\stackrel{\delta}{\longrightarrow} H^1(\calO_{Z_1})\to h^1(Z_1,\calL_{gen}^{im})\to 0$$
$\delta$ is injective, and {\it (ii)} $h^1(Z,\calL_{gen}^{im})=h^1(Z_1,\calL_{gen}^{im})$.
\end{lemma}

The lemma above has the following geometric interpretation from \cite{NNAD}:

\begin{theorem}\label{th:structure} {\bf (Structure  theorem)}\cite{NNAD}
Fix a resolution $\tX$, a cycle $Z\geq  E$  and a   Chern class $l'\in -\calS'$
as above.

(a) There exists an effective cycle $Z_1 \leq Z$, such that: {\it (i)} the map
$\eca^{l'}(Z) \to H^1(\calO_{Z_1})$ is birational onto its image,
and {\it (ii) }the generic fibres of the restriction of $r$,
$r^{im}:\im(c^{l'}(Z)) \to \im(c^{l'}(Z_1))$,  have dimension $h^1(\calO_{Z})-h^1(\calO_{Z_1})$.
(Cf. Lemma \ref{lem:MINSETS}(II).)

(b) In particular, for any such $Z_1$, the space $\im (c^{l'}(Z))$ is birationally
equivalent with an affine fibration with affine fibers of dimension
$h^1(\calO_{Z})-h^1(\calO_{Z_1})$ over $\eca^{l'}(Z_1)$.

(c) The set of effective cycles $Z_1$ with property as in {\it (a)} has a unique
minimal and a unique maximal element
denoted by $C_{min}(Z, l')$ and $C_{max}(Z, l')$.
Furthermore, $C_{min}(Z, l')$ coincides with the cohomology cycle of the pair
$(Z,\calL_{gen}^{im})$ (the unique minimal element of the set $\{0\leq Z_1\leq Z\,:\,
h^1(Z,\calL_{gen}^{im}) =h^1(Z_1,\calL_{gen}^{im})$) for the generic $\calL_{gen}^{im}\in
\im ( c^{l'}(Z))$.
\end{theorem}

In this article we want to investigate the invariants $\tau(\im (c^{l'}(Z)))$ in the cases, when the dual projective variety $(\overline{\im (c^{l'}(Z))})^{*}$  is a hypersurface.

By the results above we can assume that $Z = C_{min}(Z, l')$, indeed if $Z > C_{min}(Z, l')$ and $\overline{\im (c^{l'}(Z))} \neq \overline{\im (c^{l'}(C_{min}(Z, l')))}$, then
$\overline{\im (c^{l'}(Z))}$ is an affine fibration over $\overline{\im (c^{l'}(C_{min}(Z, l')))}$ with nontrivial fibers, and then $(\overline{\im (c^{l'}(Z))})^{*}$  is not a hypersurface.

So we can assume in the following that $Z = C_{min}(Z, l')$, which means in particular that the Abel map $c^{l'}(Z): \eca^{l'}(Z) \to \pic^{l'}(Z)$ is birational to its image and
if $\calL$ is a generic line bundle in $\im (c^{l'}(Z))$, then the cohomological cycle of $\calL$ is $Z$, so $h^1(Z_1, \calL | Z_1) < h^1(Z, \calL)$ for every integer cycle $0 \leq Z_1 < Z$.

Notice furthemore that if we investigate the ivariants $\tau(\im(c^{l'}(Z)))$, we can assume that $|Z|$ is connected.

Indeed assume otherwise that $|Z|$ is not connected and let the connected components of the cycle $|Z|$ be $|Z_1|, \cdots, |Z_i|$, where $Z = \sum_{1 \leq j \leq i} Z_j$.

We get immediately that $H^1(\calO_Z) = \oplus_{1 \leq j \leq i} H^1(\calO_{Z_j})$ and $\overline{\im(c^{l'}(Z))} = \oplus_{1 \leq j \leq i} \overline{\im(c^{l'}(Z_j))}$.

Now by lemma\ref{dir} we get that $\tau(\overline{\im(c^{l'}(Z'))}) =  \prod_{1 \leq j \leq i} \tau(\overline{\im(c^{l'}(Z_j))})$.

\section{The $\tau$ invariant of the varieties $\overline{\im(c^{l'}(Z))}$}

In the following we restrict our attention first only to generic singularities, we prove the following main theorem with the setup explained above:

\begin{theorem}
Let $\mathcal{T}$ be an arbitrary resolution graph and $\tX$ a generic singularity corresponding to it.
Let's have a Chern class $l' \in -S'$ and an integer effective cycle $Z \geq E$, such that $Z = C_{min}(Z, l')$, notice that this is 
a combinatorial condition computable from the resolution graph if the singularity is generic, and in particular we know that the map $ \eca^{l'}(Z) \to \im(c^{l'}(Z))$ is birational.
With these notations we have the following:

1) The dual projective variety of the projective clousure $\overline{\im(c^{l'}(Z))}$ has got dimension $h^1(\calO_Z)-1$.

2) Let's have the line bundle $\calL_Z = \calO_Z(K+Z)$, we have $H^0(Z, \calL_Z)_{reg} \neq \emptyset$ and it hasn't got base points at intersection points of exceptional divisors. Furthermore let's have a vertex $v \in |l'|_{*}$, so a vertex such that $(E_v, l') < 0$ , then the line bundle $\calL_Z$ hasn't got a base point on the exceptional divisor $E_v$.

3) For an arbitrary vertex $v \in \calv $ let's denote $t_v = (- Z_K+Z, E_v)$, with this notation we have got $\tau( \overline{\im(c^{l'}(Z))})= \prod_{v\in |l'|_{*}}   {t_v \choose (l', E_v)}$.
\end{theorem}
\begin{proof}

For an effective divisor $D \in \eca^{l'}(Z)$, such that $c^{l'}(Z)(D) \in \im(c^{l'}(Z))$ is smooth,  let's denote by $\Omega_D \subset H^1(\calO_Z)^* = H^0(Z, \calL_Z)$ the set of differential forms, which vanish on $T_{c^{l'}(Z)(D)}(\im(c^{l'}(Z)))$.
If $D \in \eca^{l'}(Z)$ is generic, then the map $c^{l'}(Z) : \eca^{l'}(Z) \to \im(c^{l'}(Z))$ is a submersion, which means that $T_{c^{l'}(Z)(D)}(\im(c^{l'}(Z))) = \im(T_D(  c^{l'}(Z)))$,
and so by \cite{NNI} we also have $h^1(\calO_{Z}(D)) = \dim(\Omega_D) = h^1(\calO_Z) - d_Z(l')$.

For part 1) we will prove that $H^0(Z, \calL_Z)_{reg} \neq \emptyset$ and if we have a generic element $\omega \in H^0(\calO_Z(K + Z))_{reg}$, then there is a generic divisor $D \in \eca^{l'}(Z)$ in the sense described above and another divisor $D' \in \eca^{Z - Z_K - l'}(Z)$, such that the divisor of $\omega$ is $D + D'$.

Let's see first that part 1) follows from this statement.

Indeed we have to prove that $\tau(\overline{\im(c^{l'}(Z))}) \geq 1$, so if we have a generic element in the dual space $ \omega \in H^1(\calO_Z)^*$, then there is a smooth point 
$p \in  \overline{\im(c^{l'}(Z))}$, such that $\omega$ vanishes on $T_p(\overline{\im(c^{l'}(Z))})$.
However by Seere duality we have $ H^0(\calO_Z(K + Z)) = H^1(\calO_Z)^*$, so there is a divisor $D \in \eca^{l'}(Z)$ and $D' \in \eca^{Z - Z_K - l'}(Z)$, such that the divisor of $\omega$ is $D + D'$ and the map $c^{l'}(Z) : \eca^{l'}(Z) \to \im(c^{l'}(Z))$ is a submersion in $D$.

We show that $\omega$ hasn't got a pole along the divisor $D$, or in a more precise way that $\omega$ vanishes on $\im(T_D(  c^{l'}(Z)))$.

Let's write $D = \cup_{1 \leq j \leq i} D_j$, where the cuts $D_j$ are disoint and transversal to the exceptional divisor $E$ at its smooth points.
If $D$ is an enough generic divisor, we have $\im(T_D(  c^{l'}(Z))) = \oplus_{1 \leq j \leq i} \im(T_{D_j}( c))$, so we have to prove that $\omega$ vanishes on each $ \im(T_{D_j}( c))$.
Assume that the divisor $D_j$ is supported on the smooth part of the exceptional divisor $E_u$.

We have to prove that if we have a tangent vector $v \in T_{D_j}\eca^{-E_u^*}(Z)$
and we have an arbitrary curve $f: (\bC, 0) \to \eca^{-E_u^*}(Z)$, such that $f(0) = D_j$ and $f'(0) = v$, then $\frac{d}{dt}(\omega(c^{-E_u^*}(Z)(f(0)))) = 0$.

Let's have a local chart $(x, y)$ of the space $\tX$ near the intersection point $E_u \cap D_j$ such that $E_u = (x = 0)$ and $D_j = (y = 0)$.
Let's realise the tangent vector $v$ by an aproppriate deformation of the divisor $D_j$ of the form $f(t) = [y + t \cdot \sum_{0 \leq k \leq Z_u -1} a_k \cdot x^k]$, and let's denote
$f(t) = D_t$.

We can express a representative of the differential form $\omega$ in local cordinates as $\omega = (\sum_{1 \leq i, -Z_u \leq j} a_{i, j} y^i x^j) dx \wedge dy$, so by Laufer integration formula we get:

\begin{equation*}
\frac{d}{dt}(\omega(c^{-E_u^*}(Z)(f(0)))) = \frac{d}{dt}\left(  \int_{|x|=\epsilon, \atop |y|=\epsilon}  \log \left(1+ t \cdot \frac{ \sum_{0 \leq k \leq Z_u -1} a_k \cdot x^k}{y} \right) \left(\sum_{1 \leq i, -Z_u \leq j} a_{i, j} y^i x^j \right) dx\wedge dy  \right).
\end{equation*}

However this is zero, which can be easily seen by the residuum formula.

We will prove first that if $D \in \eca^{l'}(Z)$ is a generic divisor, and we denote the distinct intersection points of $D$ with some exceptional divisor $E_v$ by $p_1, p_2, \cdots, p_m$ (where $m = (l', E_v)$), then $H^0(\calO_Z(Z+K - D))_{reg} \neq \emptyset$, and the line bundle
$\calO_Z(Z+ K - D)$ hasn't got a base point at the points $p_1, \cdots, p_m$.

Notice that $H^0(\calO_Z(Z+K))_{reg} \neq \emptyset$ also follows from this statement.

Suppose first that $H^0(\calO_Z(Z+ K - D))_{reg} = \emptyset$, this means that $H^0(\calO_Z(Z+K - D)) = H^0(\calO_Z'( Z'+K - D))$ for some effective integer cycle $0 \leq Z' < Z, Z' $, which means that $h^1(\calO_Z(D)) = h^1(\calO_{Z'}(D))$, however this is impossible by the assumption $Z = C_{min}(Z, l')$.

By simmetry we only have to prove that the line bundle $\calO_Z(Z+K - D)$ hasn't got a base point at the point $p_1$.

There are two cases, assume first that $Z_v > 1$:

Let's blow up $E_v$ at the point $p_1$ and let's denote the blow up map by $\pi_{p_1}$ and the new exceptional divisor by $E_{p_1}$ and $Z_{p_1} = \pi_{p_1}^*(Z) + (Z_v - 1) E_{p_1}$ and let $D'$ be the strict transform of the pullback of $D$, so $\pi_{p_1}^*(D) = D' + E_{p_1}$.

We should prove that $H^0(\calO_{Z_{p_1}}(\pi_{p_1}^*(Z+K - D)))_{reg} \neq \emptyset$.

Notice that we have $\calO_{Z_{p_1}}(\pi_{p_1}^*(Z+K - D))= \calO_{Z_{p_1}}(Z_{p_1} + K_{p_1} - D' - E_{p_1})$, so it means we have to prove $H^0(\calO_{Z_{p_1}}(Z_{p_1} + K_{p_1} - D' - E_{p_1}))_{reg} \neq \emptyset$.

This is equivalent to that there isn't an integer effective cycle $Z' < Z_{p_1}$, such that $h^1(\calO_{Z_{p_1}}( D' + E_{p_1})) = h^1(\calO_{Z'}( D' + E_{p_1}))$, or equivalently for every cycle $Z' < Z_{p_1}$ one has $h^1(\calO_{Z' }(D' + E_{p_1})) < h^1(\calO_{Z_{p_1}}(D' + E_{p_1})) = h^1(Z, D)$.

If for some $u \in \calv$ one has $Z'_u < Z_u$, then we know it, since by the assumptions on $Z$ for every cycle $Z'' < Z$ one has $h^1(\calO_{Z''}( D)) < h^1(\calO_Z(D))$.

Now we know that $Z_{p_1 } \geq E_{p_1}$ because of $Z_v > 1$ and it remains to prove that $h^1(\calO_{Z_{p_1} - E_{p_1}}(D' + E_{p_1})) < h^1(\calO_{Z_{p_1}}(D' + E_{p_1})) = h^1(\calO_Z( D))$.

Notice that since the map $ c^{l'}(Z) :  \eca^{l'}(Z) \to \im(c^{l'}(Z))$ is birational and $D \in \eca^{l'}(Z)$ is a generic divisor, the line bundle $\calO_Z(D)$ has got base points at the points $p_1, \cdots p_m$, so we have $H^0(\calO_{Z_{p_1}}( D' + E_{p_1}))_{reg} = \emptyset$.

On the other hand it is obvious that $H^0(\calO_{Z_{p_1}- E_{p_1}}(D'))_{reg} \neq \emptyset$.

In the following we have again two cases:

Assume first that $H^0(\calO_{Z_{p_1} - E_{p_1}}(D' + E_{p_1}))_{reg} \neq \emptyset$, then we can use Theorem\ref{relgen} to obtain $h^1(\calO_{Z_{p_1} - E_{p_1}}(D' + E_{p_1})) = h^1(\calO_{Z^*}(D' + E_{p_1}))$, where $Z^*$ is the cycle with same coefficents as $Z$, but on the blowup singularity (notice, that $D'$ is also a generic divisor on the blown up singularity).

By the definition of the cycle $Z$ we know that $h^1(\calO_Z(D)) = \chi(-l') - \chi(-l' + Z) + 1$.

Indeed from \cite{NNAD} we know that we have:

\begin{equation*}
h^1(\calO_Z(D)) = \max_{0 \leq Z' \leq Z}(  \sum_{1 \leq i \leq n}( \chi(l') - \chi(l' + Z_i) + D(Z_i, l') )),
\end{equation*}

where the connected components of $|Z'|$ are $|Z'_1|, \cdots, |Z'_n|$ and $D(Z_i, l') = 1$ if the Chern class $l'$ is not dominant of the cycle $Z_1$, and $D(Z_i, l') = 0$ otherwise.
Also the maximum is attained for a cycle $Z'$, such that $  D(Z_i, l') = 1$ for each $0 \leq i \leq n$.

If $Z' < Z$ were true, then $h^1(\calO_{Z}(D)) \leq h^1(\calO_{Z'}(D))$, which is impossible by the fact $Z = C_{min}(Z, l')$.

It means that we indeed have $h^1(\calO_Z(D)) = \chi(-l') - \chi(-l' + Z) + 1$.

On the other hand we have $h^1(\calO_{Z^*}( D' + E_{p_1})) = h^1(\calO_{Z^*}) - \dim(\im(c^{l'}(Z^*)))$, because the restriction of the divisor $D' + E_{p_1}$ to the cycle $Z^*$ is a generic divisor in $\im(c^{l'}(Z^*))$.

This means that there is a cycle $Z' \leq  Z^*$, such that $h^1(\calO_{Z^*}( D' + E_{p_1})) = \sum_{1 \leq i \leq n} (\chi(-l') - \chi(-l' + Z'_i) + D(Z'_i, l'))$, where
the connected components of $|Z'|$ are $|Z'_1|, \cdots, |Z'_n|$.

Let's look at the cycle $Z'' \leq Z$ which has got the same coefficients as $Z'$, but on the singularity before blown up, and let's denote the connected components of $|Z''|$ by $|Z''_1|,  \cdots, |Z''_n|$.

We immediately get the following:

\begin{equation*}
h^1(\calO_{Z^*}( D' + E_{p_1} )) = \sum_{1 \leq i \leq n} (\chi(-l') - \chi(-l' + Z'_i) + D(Z'_i, l')) \leq \sum_{1 \leq i \leq n} (\chi(-l') - \chi(-l' + Z''_i) + D(Z''_i, l')).
\end{equation*}

On the other hand by the condition $Z = C_{min}(Z, l')$ we know that $\sum_{1 \leq i \leq n} (\chi(-l') - \chi(-l' + Z''_i) + D(Z''_i, l')) \leq \chi(-l') - \chi(-l' + Z) + 1$ and equality happens if and only if $Z = Z''$.

It means that if $Z'' < Z$, then $h^1(\calO_{Z^*}( D' + E_{p_1} )) < \chi(-l') - \chi(-l' + Z) + 1$, so we indeed get $h^1(\calO_{Z_{p_1} - E_{p_1}}(D' + E_{p_1})) <  h^1(\calO_Z( D))$.

On the other hand if $Z = Z''$ and $Z_v \geq 2$, then we have  $\chi(-l') - \chi(-l' + Z) > \chi(-l') - \chi(-l' + Z')$ and $1 \geq D(Z', l')$, which yields that $h^1(\calO_{Z_{p_1} - E_{p_1}}(D' + E_{p_1})) <  h^1(\calO_Z( D))$ and we are done.

Now let's assume in the following that $H^0(\calO_{Z_{p_1} - E_{p_1}}(D' + E_{p_1}))_{reg}= \emptyset $, notice that in this case we have $Z_{p_1} - E_{p_1} \geq E_{p_1}$.

We know that $H^0(\calO_{Z_{p_1} - 2 E_{p_1}}( D'))_{reg} \neq \emptyset$, which means that $H^0(\calO_{Z_{p_1} - 2 E_{p_1}}(D')) = H^0(\calO_{Z_{p_1} - E_{p_1}}(D' + E_{p_1}))$ and $h^1(\calO_{Z_{p_1} - E_{p_1}}(D' + E_{p_1})) = h^1(\calO_{Z_{p_1} - 2 E_{p_1}}( D')) - 1$.

Notice, that $Z_{p_1} - 2 E_{p_1}$ is an effective integer cycle on a generic resolution and $D'$ is a generic divisor on it.

Let's denote the Chern class of $D'$ by $l'' =  \pi_{p_1}(l')- E_{p_1}$, now we know that there exists a cycle $Z' \leq Z_{p_1} - 2 E_{p_1}$ with connected components $|Z'_1|, \cdots, |Z'_n|$, such that we have:

\begin{equation*}
h^1(\calO_{Z_{p_1} - 2 E_{p_1}}( D')) = \sum_{1 \leq i \leq n } (\chi(-l'') - \chi(-l'' + Z'_i)  + D(Z'_i, l'')).
\end{equation*}

If there isn't any component $|Z'_i|$, such that $v \in |Z'_i|$, then one has $h^1(\calO_{Z_{p_1} - 2 E_{p_1}}(D'))  \leq h^1(\calO_Z(D))$ and we are done.

So assume in the following on the other hand that $v \in |Z'_1|$.

Let's have the cycles $Z''_2, \cdots Z''_n$, which have got the same coeficcients as $Z'_2, \cdots, Z'_n$, but on the singularity before the blown up.

Let's have also the cycle $Z'_1 = A' + t E_{p_1}$ and let's denote by $A''$ the cycle, which has got the same coeficcients as $A'$, but on the singularity before the blown up and let's have $Z'' = A'' + \sum_{2 \leq i \leq n} Z''_i$.

We know that $\chi(-l'') - \chi(-l'' + Z''_i)  + D(Z''_i, l'') = \chi(-l') - \chi(-l' + Z'_i)  + D(Z'_i, l')$ if $2 \leq i \leq n$.

Now if $ D(A', l') = 0$, then we have $\chi(-l') - \chi(-l' + A')  + D(A', l') \leq -1$, which yields by an easy calculation that  $\chi(-l'') - \chi(-l'' + Z'_1)  + D(Z'_1, l'') \leq 1$.

Indeed we have:
\begin{equation*}
\chi(-l'') - \chi(-l'' + Z'_1) =  \chi(-l') - \chi(-l' + A') + (A'_v - t) - \frac{(A'_v -t)(A'_v - t -1)}{2}.
\end{equation*}
From this we get that $\chi(-l'') - \chi(-l'' + Z'_1)  \leq \chi(-l') - \chi(-l' + A') + 1$ and $ D(Z'_1, l'') \leq 1$, so we indeed get $\chi(-l'') - \chi(-l'' + Z'_1)  + D(Z'_1, l'') \leq 1$.

On the other hand we have $h^1(\calO_Z(D)) = \chi(-l') - \chi(-l' + Z) + D(Z, l') > \sum_{2 \leq i \leq n} (\chi(-l') - \chi(-l' + Z'_i)  + D(Z'_i, l'))$ because of $Z \neq \sum_{2 \leq i \leq n} Z_i$.

This means that indeed we have:

\begin{equation*}
h^1(\calO_{Z_{new} - 2 E_{new}}( D')) = \sum_{1 \leq i \leq n } (\chi(-l'') - \chi(-l'' + Z'_i)  + D(Z'_i, l'')) \leq \chi(-l') - \chi(-l' + Z) + D(Z, l') = h^1(\calO_Z(D)),
\end{equation*}
and we are done.

Now assume in the following that $D(A', l') = 1$, then by the inequality $\chi(-l'') - \chi(-l'' + Z'_1)  \leq \chi(-l') - \chi(-l' + A') + 1$ we get again that $\chi(-l'') - \chi(-l'' + Z'_1)  + D(Z'_1, l'') \leq \chi(-l') - \chi(-l' + A'')  + D(A'', l') + 1$.   

On the other hand we know that $h^1(\calO_Z(D)) = \chi(-l') - \chi(-l' + Z) + D(Z, l') \geq \sum_{2 \leq i \leq n} (\chi(-l') - \chi(-l' + Z'_i)  + D(Z'_i, l')) + \chi(-l') - \chi(-l' + A'')  + D(A'', l')$ and equality happens if and only if $Z'' = Z$.

If $Z'' < Z$ it yields the statement immediately, so assume that $A'' = Z'' = Z$ in the following.

In this case $\chi(-l'') - \chi(-l'' + Z')  + D(Z', l'') \leq \chi(-l') - \chi(-l' + Z)  + 1$, so this yields the statement again.

Indeed we have $D(Z', l'') \leq 1$ and on the other hand $\chi(-l'') - \chi(-l'' + Z') = -(l'', Z') - \chi(Z') = -(\pi_p^*(l') - E_p , Z') - \chi(Z')$, which means that:

\begin{equation*}
\chi(-l'') - \chi(-l'' + Z') =  \chi(-l') - \chi(-l' + Z) + (Z_v - t) - \frac{(Z_v -t)(Z_v - t -1)}{2}.
\end{equation*}

This indeed yields $\chi(-l'') - \chi(-l'' + Z')  + D(Z', l'') \leq \chi(-l') - \chi(-l' + Z)  + 1$ since $t \leq Z_v - 3$.

Assume in the following on the other hand that $Z_v= 1$:

Let's have a generic divisor $D \in \eca^{l'}(Z)$, we know that $H^0(\calO_Z(Z + K - D))_{reg} \neq \emptyset$, so let's have a divisor $D' \in \eca^{Z-Z_K - l'}(Z)$, such that $\calO_Z(D + D') = \calO_Z(K + Z)$
and $D$ and $D'$ are disjoint on the exceptional divisors $E_u$, where $Z_u \geq 2$.

We know that such a divisor $D'$ exists by the fact that the line bundle $\calO_Z(Z + K - D)$ hasn't got a base point at the intersection points of $D$ with $E_u$, where $Z_u \geq 2$.

Let's have a section $\omega_1 \in H^0(\calO_Z(K + Z))_{reg}$, such that the divisor of $\omega_1$ is $|\omega_1| = D + D'$, and let's have a generic section $\omega_2 \in H^0(\calO_Z(K + Z))_{reg}$, such that $|\omega_2| \cap D = \emptyset$.  

By the local value distribution from $1$-dimensional complex analysis one easily gets that if $t \in (\bC, 0)$ is enough small, then we can write $ |\omega _1 + t \cdot \omega_2|= D_t + D'_t$,
where $|D_t| \cap |D'_t| = \emptyset$ and $D_t$ is close to $D_0 = D$ in $ \eca^{l'}(Z)$ if $t$ is enough small.

We know that if $t$ is enough small and $D \in \eca^{l'}(Z)$ is a generic divisor, then also $D_t$ is a generic divisor and we know that the line bundle $\calO_{Z}(Z + K-D_t)$ hasn't
got base points in $|D_t|$, which yields our statement.

We have proved that if $D \in \eca^{l'}(Z)$ is a generic divisor, and we denote the distinct intersection points of $D$ with some exceptional divisor $E_v$ by $p_1, p_2, \cdots, p_m$ (where $m = (l', E_v)$), then $H^0(\calO_Z(Z+K - D))_{reg} \neq \emptyset$, and the line bundle
$\calO_Z(Z+K - D)$  hasn't got a base point at the points $p_1, \cdots, p_m$.

Let's have a divisor $D'$ such that $|D| \cap |D'| = \emptyset$ and $\calO_Z(Z+K) = \calO_Z(D + D')$, this means, there is a section $s \in H^0( Z, Z+K)_{reg}$, such that the divisor of $s$ is $D + D'$.

Now let's have a generic section $s' \in H^0( \calO_Z(Z+K))_{reg}$, and let's have the map $f : (\bC, 0) \to H^0( \calO_Z(Z+K))_{reg}$ given by $f(t) = t \cdot s' + s$.
If we denote the map $H^0( \calO_Z(Z+K))_{reg} \to \eca^{-Z_K + Z}(Z)$ by $g$, then we have the map $g \circ f :  (\bC, 0) \to \eca^{-Z_K + Z}(Z)$, where we have $g \circ f(0) = D + D'$.

For a small enough number $t \in  (\bC, 0), t \neq 0$ we know that $t \cdot s' + s$ is a generic section in $ H^0( \calO_Z(Z+K))_{reg}$ and since $| D| \cap |D'| = \emptyset$ we know that there exist divisors $D(t), D'(t)$, such that $D(0) = D, D'(0) = D'$ and $D(t) \cap D'(t) = \emptyset$
and $g \circ f(t)= D(t) + D'(t)$. 

We also know that for small value of $t$ $D(t)$ is a generic divisor in $\eca^{l'}(Z)$, so part 1) is proved.

Notice that the fist statement of part 2) is immediate from lemma\ref{baseI} since we have proved above that $H^0(\calO_{Z}( Z+K))_{reg} \neq \emptyset$, so it follows that the 
line bundle $\calO_{Z}( Z+K)$ hasn't got base points at intersection points of exceptional divisors.

Notice that similarly the second statement of part 2) follows in the case, when $Z_v = 1$ from lemma\ref{baseII} since $H^0(\calO_{Z}( Z+K))_{reg} \neq \emptyset$.

In the following we prove the second statement of part 2) in the case $Z_v > 1$:

In fact we will prove somewhat more in the following lemma, which we state here in its full generality, so we can use it also in the proof of part 3):

\begin{lemma}\label{ind}
Let's have an arbitrary resolution graph $\mathcal{T}$ and a generic singularity $\tX$ corresponding to it.

Let's have furthermore an integer effective cycle $Z \geq E$ and a cycle $Z' \leq Z$ and an arbitrary vertex $v \in \calv$, and assume that $Z_v > 1$ and $Z'_v \geq 1$.

Let's blow up the divisors $E_u, u \in |Z'|$ in $r_u$ generic points $q_{u, 1}, q_{u, 2}, \cdots, q_{u, r_u}$ and let the new divisors be $E_{u, 1}, \cdots, E_{u, r_u}$ and let's denote $l = \sum_{u \in \calv, 1 \leq i \leq r_u} E_{u, i}$ and $Z_{new} = \pi^*(Z) - l$ and $Z'_{new} =  \pi^*(Z') - l$.

Assume that $Z' = Z$ and $H^0(\calO_Z(Z+K))_{reg} \neq \emptyset$ or if $Z' \neq Z$, then $|Z'| \neq |Z|$ and there exists a vertex $s \in |Z| \setminus |Z'|$, such that $s$ is a neighbour to the subgraph $|Z'|$ and for every vertices $w  \neq s, v$ on the unique string between $s$ and $v$
we have $r_w = 0$.

Assume furthermore that $H^0(\calO_{Z'_{new}}(K_{new} + Z_{new} - l))_{reg} \neq \emptyset$ and the dimension of the image of the map $H^0(\calO_{Z'_{new}}( K_{new} + Z_{new} - l)) \to H^0(\calO_{E_v}(K_{new} + Z_{new} - l))$ is bigger, than $1$.

Then the line bundle $\calO_{Z'_{new}}(K_{new} + Z_{new} - l)$ hasn't got a base point on the regular part of $E_v$.
\end{lemma}
\begin{proof}

We will prove this statement by induction on $h^1(\calO_{Z'})$, if $h^1(\calO_{Z'}) = 0$, then this is trivial.
So suppose in the following that $h^1(\calO_{Z'}) = k$ and we know the statement if $h^1(\calO_{Z'}) \leq k-1$.

Assume to the contrary that for a generic singularity $\tX$ the line bundle $\calL := \calO_{Z'_{new}}(K_{new} + Z_{new} - l)$ has got a base point $p$ on the regular part of $E_v$ with multiplicity $m$.

This means that for a generic section $s \in \im(H^0(\calL) \to H^0(E_v, \calL | E_v))$ the section $s$ vanishes in $p$ of order $m$.

In the following it's easy to see that we can assume that $h^1(\calO_{Z'}) > h^1(\calO_{Z''})$ for every effective integer cycle $Z'' < Z'$.

Indeed if $Z = Z'$, then we have $H^0(\calO_Z(Z+K))_{reg} \neq \emptyset$ and it indeed yields that $h^1(\calO_Z) > h^1(\calO_{Z''})$ for every cycle $Z'' < Z$, where $Z \neq Z''$.

Assume on the other hand that $Z \neq Z'$ and $h^1(\calO_{Z''}) = h^1(\calO_{Z'})$ and $h^1(\calO_{Z''}) > h^1(\calO_{Z'''})$ for every cycle $0 \leq Z''' < Z''$, then we can restrict everything to the cycle $Z''$.

We get that $ H^0(\calO_{Z''_{new}}( K_{new} + Z_{new} - l))_{reg} \neq \emptyset$ and the dimension of the image of the map $H^0(\calO_{Z''_{new}}(  K_{new} + Z_{new} - l)) \to H^0(\calO_{E_v}( K_{new} + Z_{new} - l))$ is bigger, than $1$, and the line bundle 
$\calO_{Z''_{new}}( K_{new} + Z_{new} - l)$ has got the base point $p$ of multiplicity $m$ on the regular part of $E_v$, in particular we get $Z'' \geq E_v$.

The other conditions of our lemma holds trivially, indeed there was a vertex $s \in |Z| \setminus |Z'|$, such that $s$ is a neighbour to the subgraph $|Z'|$ and for every vertices $w  \neq s, v$ on the unique string between $s$ and $v$ we have $r_w = 0$.

If we have the uniqe vertex $s'$ on the string between $s$ and $v$, such that $s'$ is a neighbour to the subgraph $|Z''|$, then we also get that for every vertices $w  \neq s', v$ on the unique string between $s'$ and $v$ we have $r_w = 0$.

This indeed means that the conditions of the lemma holds for the cycles $Z'', Z$.

So this means that we can assume that $Z'' = Z'$, which also means that $\chi(Z'') > \chi(Z')$ if $Z'' < Z'$.

There are two cases, assume first in the following that  $Z'_v = Z_v$, in particular this means, that $Z'_v \geq 2$:

Let's blow up $E_v$ sequentially in generic points $t =Z'_v -1$ times, starting with $q$, and let the new divisors be $E_{v'_1}, \cdots, E_{v'_t}$, and let $Z'_1 = \pi^*(Z'_{new}) - \sum_{1 \leq i \leq t} i \cdot E_{v'_i}$ and  $Z_1 = \pi^*(Z_{new}) - \sum_{1 \leq i \leq t} i \cdot E_{v'_i}$, we know that $h^1(\calO_{Z'_1}) = h^1(\calO_{Z'_{new}}) = k$.

We know that $e_{Z'_1}( v'_t)  > 0$, because we have blown up $E_v$ sequentially in generic points and there is a differential form in $H^1(\calO_{Z'})^*$, which has got a pole of order $Z'_v$ along the exceptional divisor $E_v$.

Let the new vertex set of the blown up singularity be $\calv_1$, and let's look at the line bundle $\pi^*(\calL) = \calO_{Z'_1}(K_{1} + Z_1 - l)$, we know that it has got a base point at $p$ with multiplicity $m$ on $E_v$ and we have $H^0(Z'_1, \pi^*(\calL))_{reg} \neq \emptyset$.

Let's denote the restriction of the cycle $Z'_1$ to the vertex set $\calv_1 \setminus E_t$ by $Z''_{1}$, where notice that $h^1(\calO_{Z''_{1}}) < h^1(\calO_{Z'_1}) = h^1(\calO_{Z'_{new}}) = k$, since $e_{Z'_1}( v'_t)  > 0$, notice also that $Z''_{1} = Z'_1 - E_t$.

On the other hand we have $Z''_1 \leq Z_1$,  and we know that the dimension of the map $H^0(\calO_{Z''_1}(K_{1} + Z_1 - l)) \to H^0(\calO_{E_v}( K_{1} + Z_1 - l))$  is bigger then $1$, so we know that the line bundle $\calO_{Z''_1}(K_{1} + Z_1 - l)$ on the cycle $Z''_1$ hasn't got a base point on the regular part of $E_v$, because we can see easily that the conditions of the induction hypothesis hold.

Indeed we have the vertex $v'_t \in |Z_1|$, which is a neighbour of $|Z'_1|$ and for every vertices $w  \neq v'_t, v$ on the unique string between $s'$ and $v$ we have $r_w = 0$.

Let's denote $Z''_{1, p} = \pi_p^*(Z''_1) - E_{p}$  and $Z'_{1, p} = \pi_p^*(Z'_1) - E_{p}$ , where $\pi_p$ is the blowup map at the base point $p$. 

We know by Seere duality that $H^0(\calO_{Z'_1}(K_{1} + Z_1 - l)) = H^1(\calO_{Z'_1}( Z'_1 - Z_1 + l))$ and $H^0(\calO_{Z'_{1, p} - E_{p}}( \pi_p^*( K_{1} + Z_1 - l ) - E_{p})) = H^1(\calO_{Z'_{1, p} - E_{p}}(\pi_p^*(Z'_1 - Z_1 + l)))$.

We know that $H^0(\calO_{Z'_1}(K_{1} + Z_1 - l )) =  H^0(\calO_{Z'_{1, p} - E_{p}}( \pi_p^*( K_{1} + Z_1 - l ) - E_{p}))$, which means that $h^1(\calO_{Z'_1}(Z'_1 - Z_1 + l)) =  h^1(\calO_{Z'_{1, p} - E_{p}}(\pi_p^*(Z'_1 - Z_1 + l)))$.

We will prove from it that $h^1(\calO_{Z''_1}( Z''_1 - Z_1 + l) = h^1(\calO_{Z''_{1,p} - E_{p}}(\pi_p^*(Z''_1 - Z_1 + l)))$, which yields that $H^0(\calO_{Z''_1}( K_{1} + Z_1 - l)) = H^0(\calO_{Z''_{1, p} - E_{p}}(\pi_p^*( K_{1} + Z_1 - l) - E_{p}))$, so the line bundle $\calO_{Z''_1}( K_{1} + Z_1 - l)$
 has got a base point at $p$, which will be a contraditcion.

Now we have two cases, first assume that $ Z' = Z$, it means obviously that $Z'_1 = Z_1$.
In this case we know that $h^1(\calO_{Z'_1}(l)) = h^1(\calO_{Z'_{1,p} - E_{p}}(l))$, and we want to prove that $h^1(\calO_{Z''_1}(- E_{v'_t} + l)) = h^1(\calO_{Z''_{1, p} - E_{p}}( - E_{v'_t} + l))$.

Let's have the pairs $(u, i), u \in \calv, 1 \leq i \leq r_u$ for which $(Z_1) \geq E_{u, i}$, it happens exactly, when $Z_{u} > 1$.

Let's denote $\sum_{u \in \calv | Z_{u} = 1, 1 \leq i \leq r_u}  E_{u, i} = l_2$  and $\sum_{u \in \calv | Z_{u} > 1, 1 \leq i \leq r_u}  E_{u, i} = l_1$.

Since for every pair $(u, i)$, such that $Z_u > 1$ we know that $(l, E_{ u_i}) < 0$ we know that $h^1(\calO_{Z_1}(  l)) = \chi( - l) - \chi( - l_2) + h^1(\calO_{Z_1 - l_1}(l_2))$.

Similarly we have $ h^1(\calO_{Z_{1, p} - E_{p}}( l)) = \chi( - l) - \chi( - l_2) + h^1(\calO_{Z_{1, p} - E_{p}- l_1 }( l_2))$, it means that we have $h^1(\calO_{Z_1 - l_1}(l_2)) = h^1(\calO_{Z_{1, p} - E_{p}- l_1 }( l_2))$.

On the other hand we have to prove that $h^1(\calO_{Z''_1}( - E_{v'_t} + l)) = h^1(\calO_{Z''_{1, p} - E_{p}}(- E_{v'_t} + l))$, now we know that $h^1(\calO_{Z''_1}(- E_{v'_t}+ l)) = \chi(E_{v'_t} - l) - \chi(E_{v'_t} - l_2) +  h^1(\calO_{Z''_1 - l_1}(- E_{v'_t} + l_2))$.

Similarly we have $ h^1(\calO_{Z''_{1, p} - E_{p}}(-E_{v'_t} + l)) = \chi(E_{v'_t} - l) - \chi(E_{v'_t} - l_2) + h^1(\calO_{Z''_{1, p} - E_{p}- l_1}(-E_{v'_t}  + l_2))$.

It means that we have to prove the following:

\begin{equation*}
 h^1(\calO_{Z''_1 - l_1}(- E_{v'_t} + l_2)) = h^1(\calO_{Z''_{1, p} - E_{p}- l_1}(-E_{v'_t}  + l_2)).
\end{equation*}

Let's look at the following exact sequence:

\begin{equation*}
0 \to  H^0(\calO_{Z''_1 - l_1}(- E_{v'_t} + l_2)) \to H^0(\calO_{Z_1 - l_1}(l_2)) \to  H^0( \calO_{E_{v'_t}}).
\end{equation*}

We know that the map $  H^0(\calO_{Z_1 - l_1}(l_2)) \to  H^0( \calO_{E_{v'_t}})$ is surjective, so we get that $h^0(\calO_{Z''_1 - l_1}(- E_{v'_t} + l_2)) =  h^0(\calO_{Z_1 - l_1}(l_2))  - 1$, which yields $h^1(\calO_{Z''_1 - l_1}(- E_{v'_t} + l_2)) =  h^1(\calO_{Z_1 - l_1}(l_2))$.

Similary let's look at the following exact sequence:

\begin{equation*}
0 \to H^0(\calO_{Z''_{1, p} - E_{p}- l_1}( l_2 - E_t) ) \to  H^0(\calO_{Z_{1, p} - E_{p}- l_1}( l_2))  \to  H^0(\calO_{E_{v'_t}}).
\end{equation*}

We know that the map $H^0(\calO_{Z_{1, p} - E_{p}- l_1}( l_2))  \to  H^0(\calO_{E_{v'_t}})$ is surjective, so we get that  $ h^0(\calO_{Z''_{1, p} - E_{p}- l_1}( - E_{v'_t} + l_2) )  =    h^0(\calO_{Z_{1, p} - E_{p}- l_1}( l_2)) - 1$, 
which yields $ h^1(\calO_{Z''_{1, p} - E_{p}- l_1}(- E_{v'_t} + l_2) )  =    h^1(\calO_{Z_{1, p} - E_{p}- l_1}( l_2))$.

These two equations immediately prove the claim in the case $ Z' = Z$.

Assume in the following, that $Z' \neq Z$, this means by our condition, that $| Z|$ is strictly bigger, than $| Z'|$ and there is a vertex $s \in |Z| \setminus |Z'|$, such that $s$ is a neighbour of the
subgraph $|Z'|$ and for every vertices $w \neq s, v$ on the unique string between $s$ and $v$ we have $r_w = 0$.

Let's have the pairs $(u, i), u \in \calv, 1 \leq i \leq r_u$ for which $(Z'_1) \geq E_{u, i}$, it happens exactly, when $Z'_{u} > 1$.

Let's denote $\sum_{u \in \calv | Z'_{u} = 1, 1 \leq i \leq r_u}  E_{u, i} = l_2$  and $\sum_{u \in \calv | Z'_{u} > 1, 1 \leq i \leq r_u}  E_{u, i} = l_1$.

We know as before that $h^1(\calO_{Z'_1}(Z'_1 - Z_1 + l)) = \chi(-Z'_1 +Z_1 - l) - \chi(-Z'_1 + Z_1 - l_2) + h^1(\calO_{Z'_1 - l_1}( Z'_1 - Z_1 + l_2))$.

Similarly we have $ h^1(Z'_{1, p} - E_{p}, \pi_p^*(Z'_1 - Z_1) + l) = \chi(-Z'_1 + Z_1 - l) - \chi(-Z'_1 + Z_1 - l_2) +  h^1(\calO_{Z'_{1, p} - E_{p}- l_1}(\pi_p^*(Z'_1 - Z_1) + l_2))$.

It means that we know:

\begin{equation*}
h^1(\calO_{Z'_1 - l_1}( Z'_1 - Z_1 + l_2)) = h^1(\calO_{Z'_{1, p} - E_{p}- l_1}(\pi_p^*(Z'_1 - Z_1) + l_2)).
\end{equation*}
On the other hand we have to prove that $h^1(\calO_{Z''_1}( Z''_1 - Z_1 + l)) = h^1(\calO_{Z''_{1, p} - E_{p}}(\pi_p^*(Z''_1 - Z_1) + l))$.

We know that $h^1(\calO_{Z''_1}( Z''_1 - Z_1 + l)) = \chi(-Z''_1 + Z_1 - l) - \chi(-Z''_1 + Z_1 - l_2) + h^1(\calO_{Z''_1 - l_1}( Z''_1 - Z_1 + l_2))$.
Similarly we have $ h^1(Z''_{1, p} - E_{p}, \pi_p^*(Z''_1 - Z_1) + l) = \chi(-Z''_1 + Z_1 - l) - \chi(-Z''_1 + Z_1 - l_2) + h^1(\calO_{Z''_{1, p} - E_{p}- l_1 }(\pi_p^*(Z''_1 - Z_1) + l_2))$.

It means that we have to prove:

\begin{equation*}
 h^1(\calO_{Z''_1 - l_1}( Z''_1 - Z_1 + l_2)) = h^1(\calO_{Z''_{1, p} - E_{p}- l_1}(\pi_p^*(Z''_1 - Z_1) + l_2)).
\end{equation*}

Let's have the string $u_1 = s, u_2, \cdots , u_q = v$, then by the conditions on $s$, there is a Laufer sequence $A_1 = E_{u_2}, ..., A_{q-1} = A$, where 
$v \in |A| = (u_1, \cdots, u_q)$ and $A_{i+1} = A_i + E_{u_i}$ for $1 \leq i \leq q-1$ and furthermore $(Z_1' - Z_1 + l_2 - A_i,  E_{u_{i}}) = ( \pi_p^*(Z'_1 - Z_1) + l_2 - A_i,  E_{u_{i}}) < 0$ and $( Z''_1 - Z_1 + l_2 - A_i,  E_{u_{i}}) = ( \pi_p^*(Z''_1 - Z_1) + l_2 - A_i,  E_{u_{i}}) < 0$
and $E_{u_{j_i}} \in |Z''_1 - l_1 - A_{i-1}|$ for every $1 \leq i \leq q-1$ with the notation $A_0 = 0$.

We get that $ h^1(\calO_{Z'_1 - l_1 - A}(Z'_1 - Z_1 + l_2 - A)) = h^1(\calO_{Z'_{1, p} - E_{p} - l_1 - A}(\pi_p^*(Z'_1 - Z_1) + l_2- A))$ and we should prove that
$ h^1(\calO_{Z''_1 - l_1 - A}(Z''_1 - Z_1 + l_2 - A)) = h^1(\calO_{Z''_{1, p} - E_{p}- l_1 - A }( \pi_p^*(Z''_1 - Z_1) + l_2 - A))$.

Now notice that $(Z''_1 - Z_1 + l_2 - A, E_{v'_1}) < 0$, so we can start a Laufer sequence $B_i =  \sum_{1 \leq j \leq i} E_{v'_j}, 1 \leq  i \leq t$, such that $B_{i} = B_{i-1}+ E_{v'_i}$
and $(Z'_1 - Z_1 + l_2 - A -B_{i-1},  E_{v'_i}) < 0$.

From these Laufer sequences we get that both $ h^1(\calO_{Z'_1 - l_1 - A}(Z'_1 - Z_1 + l_2 - A)) = h^1(\calO_{Z'_{1, p} - E_{p} - l_1 - A}(\pi_p^*(Z'_1 - Z_1) + l_2- A))$ and
$ h^1(\calO_{Z''_1 - l_1 - A}(Z''_1 - Z_1 + l_2 - A)) = h^1(\calO_{Z''_{1, p} - E_{p}- l_1 - A }( \pi_p^*(Z''_1 - Z_1) + l_2 - A))$ are equivalent with
$ h^1(\calO_{Z'_1 - l_1 - A - B_t}(Z'_1 - Z_1 + l_2 - A - B_t)) = h^1(\calO_{Z'_{1, p} - E_{p} - l_1 - A - B_t}(\pi_p^*(Z'_1 - Z_1) + l_2- A - B_t))$, which proves the statement in the case $Z'_v = Z_v$.

Now let's see the more interesting case in the following so assume that $Z'_v < Z_v$:

Let's denote again $t = Z'_v - 1$ and let's blow up the exceptional divisor $E_v$ sequentially in generic points and let the new exceptional divisors be $E_{v'_1}, \cdots, E_{v'_t}$, where perhaps we have $t = 0$.
We know that every differntial form in $H^1(\calO_{Z'})^*$ has got a pole on the exceptional divisor $E_{v'_t}$ of order at most $1$.

Now let's denote $Z'_1 = \pi^*(Z'_{new}) - \sum_{1 \leq i \leq t} i \cdot E_i$ and $Z_1 = \pi^*(Z_{new}) - \sum_{1 \leq i \leq t} i \cdot E_i$ and  let the new vertex set of the blown up singularity be $\calv_1$, with this notations we have that $e_{Z'_{1}}( v'_t) > 0$.

Let's look at the line bundle $\pi^*(\calL) = \calO_{Z'_1}( K_{1} + Z_1 - l)$, we know that it has got a base point at $p$ with multiplicity $m$ on $E_v$.

Let's denote the vertex set $\calv_1 \setminus E_{v'_t}$ by $\calv_s$ and the restriction of the cycle $Z'_1$ to $\calv_s$ by $Z'_s$.

We know that if $t > 0$, then the dimension of the image of the map $H^0(Z'_s, \calL) \to H^0(E_v, \calL)$ is greater than $1$ and the induction hypothesis holds for $Z_1$ and $Z'_s$, 
so this means that the line bundle $\calL | Z'_s$ hasn't got a base point on the exceptional divisor $E_v$.

On the other hand if $t = 0$, then $E_v \notin |Z'_s|$.

Let's have a large number $N$, such that $\dim(\im(c^{-N E_{v'_t}^*}(Z'_1)))= e_{Z'_1}(v'_t)$, and let's blow up $E_{v'_t}$ in $N $ distinct generic points $p_1, \cdots p_N$, and let the new exceptional divisors be $E_{w_1}, \cdots , E_{w_N}$.

Let's denote the new singularity by $\tX_{b}$ and let's denote its subsingularity with vertex set $\calv_b \setminus w_1, \cdots, w_N$ by $\tX_u$, we have $p_g(\tX_u) = p_g(\tX_1)$ and $h^1(\calO_{Z'_u}) = h^1(\calO_{Z'_{1}})$, since none of the differential forms
in $H^1(\calO_{Z'})^*$ has got a pole along the exceptional divisors $E_{w_1}, \cdots , E_{w_N}$.

Let's denote furthermore the line bundle $ \calL_u = \calO_{Z'_u}( \pi^*( K_{1} + Z_{1} - l)) = \calO_{Z'_u}(K_{b} + \pi^*(Z_{1}) - \sum_{1 \leq i \leq N} E_{w_i} - l)$ on the cycle $Z'_u = \pi^*(Z'_{1}) - \sum_{1 \leq i \leq N} E_{w_i}$.

We know that the dimension of the image of the map $H^0(Z'_u, \calL_u) \to H^0(E_v,  \calL_u)$ is bigger then $1$, and we should prove that $\calL_u$ hasn't got a base point on the regular part of $E_v$.

Indeed, this is enough, since $h^1(\calO_{Z'_u}) = h^1(\calO_{\pi^*(Z_{1})})$ and $H^0(\calO_{\pi^*(Z'_{1})}( \pi^*( K_{1} + Z_{1} - l) )_{reg} \neq \emptyset$, so we get that $\calO_{\pi^*(Z'_{1})}( \pi^*( K_{1} + Z_{1} - l) $ also hasn't got a base point on the regular part of $E_v$, which is what we want to prove.

Now we know that $\calv_s =\calv_u \setminus v'_t$ , and the corresponding subsingularity is $\tX_s$.

We have the line bundle $\calL_s = \calL_u | Z'_s$ on the cycle $Z'_s$,  let's denote $c^1(\calL_u) = l'_u$ and $c^1(\calL_s) = l'_s$.

Notice that the coefficient of $E_{w_i}$ in $ K_{b} + \pi^*(Z_{1}) - \sum_{1 \leq i \leq N} E_{w_i} - l$ is nonzero because of $Z_v > Z'_v$.

Notice also that every differential form in $H^1(\calO_{Z'})^*$ has got a pole along the exceptional divisor $E_{v'_t}$ of order at most $1$, which means that if $D \in \eca^{-E_{v'_t}}(Z'_u)$ is an arbitrary divisor, then the line bundle $\calO_{Z'_u}(D)$ only depends on the intersection point $D \cap E_{v'_t}$.

Let's move in the following the intersection points $E_{w_i} \cap E_{v'_t}$ generically and the analytic type of the singularity as well.
Notice that if the original singularity was enough generic and we move the plumbing of the tubular neighborhood of the exceptional divisors $E_{w_i}, 1 \leq i \leq N$ with $\tX_u$ generically, then we get generic analytic types.

Notice that the restriction $\calL_u | Z'_s$ remains the same line bundle $\calL_s$ if we move the intersection points $E_{w_i} \cap E_{v'_t}$, since each divisor $D \in \eca^{-E_{v'_t}}(Z'_u)$ restricts to the zero divisor on $\tX_s$.

On the other hand we know that $\dim(\im(c^{-N E_{v'_t}^*}(Z'_1)))= e_{Z'_1}(v'_t)$ and the coefficients of $E_{w_1}, \cdots, E_{w_N}$ are positive in $ K_{b} + \pi^*(Z_{1}) - \sum_{1 \leq i \leq N} E_{w_i} - l$, which means that if we move the contact points $p_1, \cdots, p_N$, then the line bundle $\calL_u$ cover an open set in $ r_s^{-1}(\calL_s) \subset \pic^{l'_u}(Z'_u)$.

It means that for a generic choice of the contact points $p_1, \cdots, p_N$ the line bundle $\calL_u$ is a generic line bundle in $ r_s^{-1}(\calL_s )$.

We know that $H^0(Z'_u, \calL_u)_{reg} \neq \emptyset$ for generic analytic types, which means that the pair $(l'_u, \calL_s)$ is relative dominant on the cycle $Z'_u$, which means by Theorem\ref{th:dominantrel} that:

\begin{equation*}
\chi(-l'_u)  - h^1(Z'_s, \calL_s) < \min_{0 < A \leq Z'_u} \left( \chi(-l'_u + A)  - h^1((Z'_u - A)_s, \calL_s -  A)  \right).
\end{equation*}

Now we have the following lemma:

\begin{lemma}
Let's have the setup above and assume that the dimension of the image of the map $H^0(Z'_u, \calL_u) \to H^0(E_v, \calL_u)$ is more than $1$ and let $q \in E_v$ be a generic point and let's blow up $E_v$ in $q$.
Let the new divisor be $E_{q}$, the new singularity $\tX_{u, q}$ and $Z'_{u, q} = \pi_q^*(Z'_u)$.

1) Assume that $t=0$, which means that $v'_t = v$, we claim that the pair $(\pi_q^*(l'_u)- E_{q}, \calL_s)$ is relative dominant on the cycle $Z'_{u, q}$.

2) Assume that $t > 0$, which means that $v \in \calv_s$ and let's have the line bundle $\calL_{s, q} = \pi^*(\calL_s)( - E_{q})$ on $Z'_{s, q}$, we claim that $(\pi_q^*(l_u)- E_{q}, \calL_{s,q})$ is relative dominant on $Z'_{u, q}$.
\end{lemma}

\begin{proof}

By Theorem\ref{th:dominantrel} for part 1) we have to prove that for all cycles $0 < \pi_q^*(l_d) + a E_{q} \leq \pi_q^*(Z'_u)$, such that $H^0(\pi_q^*(Z'_u) - \pi_q^*(l_d) - a E_{q}, \pi_q^*(\calL_u)(- \pi_q^*(l_d) - (a+1) E_{q}))_{reg} \neq \emptyset$, where $\calL_u$ is a generic line bundle in $ r_s^{-1}(\calL_s )$ one has:

\begin{equation*}
\chi(-\pi_q^*(l'_u) + E_{q}) - h^1(Z'_s, \calL_s) < \chi(-\pi_q^*(l'_u) + (a+1) E_{q} + \pi_q^*(l_d)) - h^1((Z'_u - l_d)_s, \calL_s (-\pi_q^*(l_d) - a E_{q})).
\end{equation*}

\begin{equation*}
\chi(-l'_u)  + 1 - h^1(Z'_s, \calL_s) < \chi(-l'_u + l_d)  +  \frac{(a+1)(a+2)}{2} - h^1((Z'_u - l_d)_s, \calL_s(- l_d)) .
\end{equation*}

The condition $H^0(\pi_q^*(Z'_u) - \pi_q^*(l_d) - a E_{q}, \pi_q^*(\calL_u)( - \pi_q^*(l_d) - (a+1) E_{q}))_{reg} \neq \emptyset$ means in particular that $H^0(Z'_u - l_d,  \calL_u(- l_d))_{reg} \neq \emptyset$.

Now assume first that  $(l_d)_v >0$, we have the following exact sequence:

\begin{equation*}
0 \to H^0(Z'_u - E_v, \calL_u(- E_v)) \to H^0(Z'_u , \calL_u) \to H^0(E_v , \calL_u).
\end{equation*}

It means that by the condition of the lemma we get immediately that $\dim\left( \frac{H^0(Z'_u , \calL_u)}{H^0(Z'_u - E_v, \calL_u(- E_v))} \right) > 1$.
This means that $\dim\left( \frac{H^0(Z'_u , \calL_u)}{H^0(Z'_u -  l_d, \calL_u(- l_d))} \right) > 1$, from which it follows that:

\begin{equation*}
\chi(-l'_u)  - h^1(Z'_s, \calL_s)  + 1< \chi(-l'_u + l_d)  - h^1((Z'_u - l_d)_s, \calL_s( -l_d)).
\end{equation*}

Indeed we know that $\chi(-l'_u)  - h^1(Z'_u, \calL_u)  + 1< \chi(-l'_u + l_d)  - h^1(Z'_u - l_d, \calL_s (-l_d))$ and the pair $(l'_u, \calL_s)$ is relatively dominant on the cycle $Z'_u$, 
which means that $h^1(Z'_s, \calL_s) = h^1(Z'_u, \calL_u)$.

On the other hand $h^1((Z'_u - l_d)_s, \calL_s(-l_d)) \leq h^1(Z'_u - l_d, \calL_s (-l_d))$, which gives the claimed inequality.

The statement immediately follows in this case, since $\frac{(a+1)(a+2)}{2} \geq 0$.

On the other hand if $l_d > 0$, but $(l_d)_v = 0$, then since $0 < \pi^*(l_d) + a E_q$ we have $a \geq 0$ and so $\frac{(a+1)(a+2)}{2} \geq 1$.

Furthermore we know that $\chi(-l'_u)  + 1 - h^1(Z'_s, \calL_s) \leq \chi(-l'_u + l_d)  - h^1((Z'_u - l_d)_s, \calL_s(- l_d))$, so the statement follows again immediately.

If $l_d = 0$, then we have $a > 0$ and we have to prove:

\begin{equation*}
\chi(-l'_u)  + 1 - h^1(Z'_s, \calL_s) < \chi(-l'_u)  +  \frac{(a+1)(a+2)}{2} -  h^1(Z'_s, \calL_s).
\end{equation*}

Notice that this is also trivial, because if $a > 0$, then one has $\frac{(a+1)(a+2)}{2} > 1$.

For part 2) we have to prove that for all cycles $0 < \pi_q^*(l_d) + a E_{q} \leq \pi_q^*(Z'_u)$, such that $H^0(\pi_q^*(Z'_u) - \pi_q^*(l_d) - a E_{q}, \pi_q^*(\calL_u)(- \pi_q^*(l_d) - (a+1) E_{q}))_{reg} \neq \emptyset$, where $\calL_u$ is a generic line bundle in $ r_s^{-1}(\calL_{s})$ one has:

\begin{equation*}
\chi(-\pi_q^*(l'_u) + E_{q}) - h^1(Z'_{s, q}, \calL_{s, q}) < \chi(- \pi_q^*(l'_u) + (a+1) E_{q} + \pi_q^*(l_d)) - h^1((Z'_{u, q} - \pi_q^*(l_d) - a E_{q})_{s, q} , \pi_q^*(\calL_s)(-\pi_q^*(l_d) - (a+ 1)E_{q})).
\end{equation*}

\begin{equation*}
\chi(-l'_u)  + 1 - h^1(Z'_{s}, \calL_s) < \chi(-l'_u + l_d) + \frac{(a+1)(a+2)}{2} - h^1((Z'_{s, q} - \pi_q^*(l_d) - a E_{q})_{s, q}, \pi_q^*(\calL_s)(- \pi_q^*(l_d) - (a+ 1)E_{q}) ).
\end{equation*}

The condition $H^0(\pi_q^*(Z'_u) - \pi_q^*(l_d) - a E_{q}, \pi_q^*(\calL_u)(- \pi_q^*(l_d) - (a+1) E_{q}))_{reg} \neq \emptyset$ means in particular that $H^0(Z'_u - l_d,  \calL_u(- l_d))_{reg} \neq \emptyset$.

Notice that if $a < -1$, then $E_q \in |\pi_q^*(Z'_u) - \pi_q^*(l_d) - a E_{q}|$ and $(E_q, c_1(\pi_q^*(\calL_u)(- \pi_q^*(l_d) - (a+1) E_{q})) ) < 0$, which contradicts 
the condition $H^0(\pi_q^*(Z'_u) - \pi_q^*(l_d) - a E_{q}, \pi_q^*(\calL_u)(- \pi_q^*(l_d) - (a+1) E_{q}))_{reg} \neq \emptyset$.

It means that $a \geq -1$, assume first that $a = -1$, in this case we have $l_d \geq E_v$ and we should prove:

\begin{equation*}
\chi(-l'_u)  + 1 - h^1(Z'_{s}, \calL_s) < \chi(-l'_u + l_d) - h^1((Z'_{u, q} - \pi_q^*(l_d) +  E_{q})_{s, q},  \pi_q^*(\calL_s)(- \pi_q^*(l_d))).
\end{equation*}

\begin{equation*}
\chi(-l'_u)  + 1 - h^1(Z'_{s}, \calL_s) < \chi(-l'_u + l_d) - h^1((Z'_{u} - l_d)_s, \calL_s(- l_d)).
\end{equation*}

We have the following exact sequence:

\begin{equation*}
0 \to H^0(Z'_u - E_v, \calL_u(- E_v)) \to H^0(Z'_u , \calL_u) \to H^0(E_v , \calL_u).
\end{equation*}

It means that by the condition of the lemma we get immediately that $\dim\left( \frac{H^0(Z'_u , \calL_u)}{H^0(Z'_u - E_v, \calL_u(- E_v))} \right) > 1$.

This means that $\dim\left( \frac{H^0(Z'_u , \calL_u)}{H^0(Z'_u -  l_d, \calL_u (-l_d))} \right) > 1$, from which it follows indeed that $\chi(-l'_u)  + 1 - h^1(Z'_{s}, \calL_s) < \chi(-l'_u + l_d) - h^1((Z'_{u} - l_d)_s, \calL_s(- l_d))$ as in part 1).

We can assume in the following, that $a \geq 0$, if $l _d>0$ we know that:

\begin{equation*}
\chi(-l_u)  + 1 - h^1(Z'_{s}, \calL_s) \leq \chi(-l'_u + l_d)  - h^1((Z'_u - l_d)_s, \calL_s(- l_d)).
\end{equation*}

It means that we only have to prove:

\begin{equation*}
\frac{(a+1)(a+2)}{2} - h^1((Z'_{s, q} - \pi_q^*(l_d) - a E_{q})_{s, q}, \pi_q^*(\calL_s)(- \pi_q^*(l_d) - (a+ 1)E_{q}))  > - h^1((Z'_u - l_d)_s, \calL_s(- l_d)).
\end{equation*}

However it immediately follows from $H^0((Z'_{s, q} - \pi_q^*(l_d) - a E_{q})_{s, q}, \pi_q^*(\calL_s)(- \pi_q^*(l_d) - (a+ 1)E_{q}))  \subset H^0((Z'_u - l_d)_s, \calL_s(- l_d))$ and
$H^0((Z'_{s, q} - \pi_q^*(l_d) - a E_{q})_{s, q}, \pi_q^*(\calL_s)(- \pi_q^*(l_d) - (a+ 1)E_{q}))  \neq H^0((Z'_u - l_d)_s, \calL_s(- l_d))$, since the line bundle $  \calL_s  \otimes \calO_{(Z'_u - l_d)_s}(-l_d)$ hasn't got a base point at $p$, because $p$ is a generic point on $E_v$.

If $l_d = 0$, then we have $a > 0$ and we have to prove:

\begin{equation*}
1 - h^1(Z'_{s, q}, \pi_q^*(\calL_s)(- E_{q})) <  \frac{(a+1)(a+2)}{2} - h^1(Z'_{s, q} - aE_q , \pi_q^*(\calL_s) - (a+ 1)E_{q}) .
\end{equation*}

For this we only have to prove $ H^0(Z'_{s, q}, \pi_q^*(\calL_s)(- E_{q}))  \neq  H^0(Z'_{s, q} - aE_q , \pi_q^*(\calL_s) - (a+ 1)E_{q})$, but this immediately follows from the fact that $\calL_s$ has not got base point on $E_v$, so the generic section $s \in H^0(Z'_s, \calL_s)$ has got $(l_u, E_v)$ disjoint arrows on $E_v$, which means indeed that if $q$ is a generic point, then $H^0(Z'_{s, q}, \pi_q^*(\calL_s)(- E_{q}))_{reg} \neq \emptyset$.

\end{proof}

Let's finish the proof of our main lemma with the help of the lemma before.

We have to prove that a generic line bundle $\calL_u  \in r_s^{-1}(\calL_s )$ hasn't got a base point on the exceptional divisor $E_v$.

Let's look at the space $\eca^{l'_u, \calL_s}(Z'_u) \subset \eca^{l'_u}(Z'_u)$ consisting of divisors $D \in \eca^{l'_u}(Z'_u)$, such that $\calO_{Z'_u}(D) | Z'_s = \calL_s$.

Now assume to the contrary that a generic line bundle $\calL_u \in  r_s^{-1}(\calL_s )$ has got a base point on $E_v$, this means that there is open subset $U \in \eca^{l'_u, \calL_s}(Z'_u) $ and a map $ f: U \to E_v$, such that $f(D) \in |D|$ is a base point of the line bundle $\calO_{\tX_u}(D)$ on the regular part of $E_v$.

Now let's have a generic divisor $D \in U$ which has got disjoint arrows on the regular part of $E_v$ and let's blow up $E_v$ at the generic point $f(D) = q$, let the new exceptional divisor be $E_{q}$.

Let's denote the new singularity by $ \tX_{q}$ and the line bundle $\calL_q = \pi_q^*(\calL_u) - E_{q}$, and the cycle $Z_q = \pi^*(Z'_u) - E_{q}$.

We know that $\eca^{l'_u, \calL_s}(Z'_u) \cap \eca^{\pi_q^*(l'_u) -E_{q}}(Z_q) =  \eca^{  \pi_q^*(l'_u) -E_{q}, \calL_s}(Z_q)$  if $t= 0$ and
$\eca^{l'_u, \calL_s}(Z'_u) \cap \eca^{\pi_q^*(l'_u) -E_{q}}(Z_q) = \eca^{\pi_q^*(l'_u) -E_{q}, \calL_{s, q}}(Z_q)$ if $t > 0$.

We can assume that $D$ is so generic such that the intersection $ \eca^{l'_u, \calL_s}(Z'_u) \cap \eca^{\pi_q^*(l'_u) -E_{q}}(Z_q)$ is transversal in $D$.

Assume first that $t = 0$:

If we look at a small open neighborhood $U_1 \subset \eca^{  \pi_q^*(l'_u) -E_{q}, \calL_s}(Z_q)$, then we get that if $D' \in U_1$, then $f(D') = f(D)$.

This means that $f(D) = f(D')$ is a base point of the line bundle $\calO_{\tX_u}(D')$, which means $h^1(\calO_{Z_q}( D')) = h^1(\calO_{Z'_u}(D')) + 1 = h^1(Z'_s, \calL_s) + 1$.

It means that the map $ \eca^{  \pi_q^*(l'_u) -E_{q}, \calL_s}(Z_q) \to r^{-1}(\calL_s) \subset \pic^{\pi_q^*(l'_u) -E_{q}}(Z_q)$ cannot be a submersion at any of the points $D' \in U_1$, because otherwise we would have $h^1(\calO_{Z_q}(D')) =  h^1(Z_s, \calL_s)$.

Indeed if the map $\eca^{  \pi_q^*(l'_u) -E_{q}, \calL_s}(Z_q) \to r^{-1}(\calL_s)$ were a submersion at a point $D' \in U_1$, then $T_{\calO_{Z_q}(D')}(  r^{-1}(\calL_s)) \subset  \im(T_{D'}(c^{\pi_q^*(l'_u) -E_{q}}(Z_q))$, which means that $h^1(\calO_{Z_q}(D')) = h^1(\calO_{Z_q}) - \dim(\im(T_{D'}(c^{\pi_q^*(l'_u) -E_{q}}(Z_q))) =  h^1(\calO_{Z_s}) - \dim(\im(r \circ T_{D'}(c^{\pi_q^*(l'_u) -E_{q}}(Z_q)))$.

On the other hand we know that $\im(r \circ T_{D'}(c^{\pi_q^*(l'_u) -E_{q}}(Z_q)) = \im(T_{D'}(c^{l'_s}(Z_s)))$, which indeed would give that $h^1(\calO_{Z_q}(D')) =  h^1(Z_s, \calL_s)$.

However we know that $\eca^{  \pi_q^*(l'_u) -E_{q}, \calL_s}(Z_q)$ is irreducible from \cite{R} and by our previous lemma we know that the map $\eca^{  \pi_q^*(l'_u) -E_{q}, \calL_s}(Z_q) \to r^{-1}(\calL_s)$ is dominant, which is a contradiction, since $U_1 \subset \eca^{  \pi_q^*(l'_u) -E_{q}, \calL_s}(Z_q)$ is open.

Now assume in the following that $t > 0$, in this case we have $v \in \calv_s$:

If we look at a small open neighborhood $U_1 \subset  \eca^{\pi_q^*(l'_u) -E_{q}, \calL_{s, q}}(Z_q)$ of $D$, then we get that if $D' \in U_1$, then $f(D') = f(D)$.

This means that $f(D) = f(D')$ is a base point of the line bundle $\calO_{\tX_u}(D')$, which means $h^1(Z_q, D') = h^1(Z'_u, D') + 1 = h^1(Z_s, \calL_s) + 1$.

It means that the map $ \eca^{\pi_q^*(l'_u) -E_{q}, \calL_{s, q}}(Z_q)  \to r^{-1}(\calL_{s, q}) \subset \pic^{\pi^*(l'_u) -E_{new}}(Z_b)$ cannot be a submersion in any of the points $D' \in U_1$, because otherwise we would have similarly as in the previous case that $h^1(\calO_{Z_q}( D')) =  h^1(Z'_{s, q}, \calL_{s, q}) =  h^1(Z'_s, \calL_s)$, where the second equality follows from the fact that the line bundle  $\calL_s$ hasn't got a base point on the exceptional divisor $E_v$ at the generic point $q$.

On the other hand by our previous lemma we know that the map $\eca^{\pi_q^*(l'_u) -E_{q}, \calL_{s, q}}(Z_q)  \to r^{-1}(\calL_{s, q}) $ is dominant, and this is a contradiction, since $U_1 \subset \eca^{\pi_q^*(l'_u) -E_{q}, \calL_{s, q}}(Z_q)  $ is open.
\end{proof}

Now we want to prove part 3), so first let's see that $\tau( \overline{\im(c^{l'}(Z))}) \leq \prod_{v\in |-l'|}   {t_v \choose (l', E_v)}$, where $t_v = (K+Z, E_v)$ for an arbitrary vertex $v \in |Z|$.

Let's have a generic differential form $w \in H^0(\calO_Z(K+Z))_{reg}$, where $w$ has got $t_v$ disjoint transvesal arrows $D_{v, i}, 1 \leq i \leq t_v$ along the regular parts of exceptional divisors $E_v$, such that $(l', E_v) > 0$, because the line bundle $\calO_Z(K+Z)$ hasn't got any base points on these exceptional divisors.

Furthermore if $w$ is enough generic, then we have exactly $m = \tau(\im(c^{l'}(Z)))$ distinct smooth point of $\im(c^{l'}(Z))$, $p_1, p_2, \cdots, p_m$ such that $w$ vanishes on the tangent spaces $T_{p_i}(\im(c^{l'}(Z)))$.

Also if $w$ is enough generic, then we can also assume that $p_i$ are also generic points of $\im(c^{l'}(Z))$ in the sense that $\dim(c^{l'}(Z)^{-1}(p_i)) = 0$, the map $c^{l'}(Z) : \eca^{l'}(Z) \to \pic^{l'}(Z)$ is a submersion at the unique point $D_i \in c^{l'}(Z)^{-1}(p_i)$ and the unique divisor $D_i \in c^{l'}(Z)^{-1}(p_i)$ has got $(l', E_v)$ disjoint arrows at the regular part of the divisors $E_v, v \in |Z|$.

Since the map $c^{l'}(Z) : \eca^{l'}(Z) \to \pic^{l'}(Z)$ is a submersion at the unique point $D_i \in c^{l'}(Z)^{-1}(p_i)$, one has $\im(T_D(c^{l'}(Z))) = T_{p_i}(\im(c^{l'}(Z)))$ and the points $p_i$ are regular values of the map $c^{l'}(Z)$.

This means that the differential form $w$ has got no pole on each $D_i \in c^{l'}(Z)^{-1}(p_i)$.

Notice however that the differential form $w$ has got a pole of order $Z_v$ on each exceptional divisor $E_v, v \in |Z|$ and disjoint transvesal arrows $D_{v, i}, 1 \leq i \leq t_v$ along the regular parts of exceptional divisors $E_v, v \in |Z|$.

Let's have the following lemma:

\begin{lemma}\label{cut}
With the setup above if $D_i \in c^{l'}(Z)^{-1}(p_i)$ is the unique divisor, which have got $(l', E_v)$ disjoint arrows at the regular part of the divisors $E_v, v \in |Z|$ and $u \in |Z|$, then $D_i$ has got $(l', E_u)$ arrows $D_{u, a_1}, \cdots, D_{u, a_{(l', E_u)}}$ supported on the exceptional divisor $E_u$, where $1 \leq a_1 < \cdots < a_{(l', E_u)} \leq t_u$.
\end{lemma}

\begin{proof}

Let the divisor $D_i \in |p_i|$ have an arrow at the exceptional divisor $E_u, u \in |Z|$ and let's denote it by $D'$, we have to prove that $D'$ is one of the arrows $D_{u, i}, 1 \leq i \leq t_u$.

We know that the differential form $w$ has got a pole of order $Z_u$ on the exceptional divisor $E_u$ and $w$ hasn't got a pole along $D'$.
It means that $D' \cap E_u = D_{u, i} \cap E_u$ for some $1 \leq i \leq t_u$, if $Z_u = 1$ then this proves the statement.

We claim in the following that $D' = D_{u, i}$:

So assume that $Z_u \geq 2$ and let's blow up $E_u$ at the point $D' \cap E_u$ and let the new exceptional divisor be $E_1$ and let the strict transforms of the divisors $D',  D_{u, i}$ be $D'_1,  D_{u, i, 1}$.

We know that the differential form $w$ has got a pole of order $Z_u - 2$ on $E_1$ and $w$ hasn't got a pole along the exceptional divisor $D'$, which means that $w$ must vanish on $D'_1$, so $D'_1 \cap E_1 = D_{u, i, 1} \cap E_1$ since $Z_u - 2 \geq 0$.

We prove by induction that if $j \leq Z_u - 1$ and we blow up the exceptional divisor $E_u$ sequentially $j$ times along the divisor $D'$ and we denote the strict transforms of $D', D_{u, i}$ by $D'_j, D_{u, i, j}$, and the new exceptional divisors by $E_1, \cdots E_j$, then $D'_j \cap E_j = D_{u, i, j} \cap E_j$.

If we apply this for $j =  Z_u - 1$, then we get $D' = D_{u, i, j}$ and it proves the statement.

Now if $j \leq Z_u - 1$ we know that $w$ has got a pole on $E_j$ of order $Z_u - 2j$, however $w$ hasn't got a pole on the divisor $D'$, which means that $w$ has got a pole on the divisor $D'_j$ of order at most $-j$.
We know, that $j \leq Z_u - 1$, so we have $Z_u - 2j > - j$ and this means, that $w$ should have an arrow at $E_j \cap D'_j$ and we indeed get $E_j \cap D'_j = E_j \cap D_{u, i, j}$ and it proves the statement of the lemma.

\end{proof}

\begin{remark}
Notice that there is a much easier statement in the other direction, namely if $D = \sum_{v \in |Z|, 1 \leq i \leq (l', E_v)} D_{v, a_{v, i}}$, where
$1 \leq a_{v, 1} < \cdots < a_{v, (l', E_v)} \leq t_v$, then the differential form $w$ hasn't got a pole on the divisor $D$.
\end{remark}

We got that if $w$ is enough generic, such that $p_i$ are also generic points of $\im(c^{l'}(Z))$ in the sense that the unique divisor $D_i \in |p_i|$ has got $(l', E_v)$ disjoint arrows at the regular part of the divisors $E_v, v \in |Z|$ and the points $p_i$ are regular values of the map $c^{l'}(Z)$
,then if $D_i \in |p_i|$, then  $D$ has got $(l', E_u)$ arrows $D_{u, a_1}, \cdots, D_{u, a_{(l', E_u)}}$ supported on the exceptional divisor $E_u$, where $1 \leq a_1, \cdots, a_{(l', E_u)} \leq t_u$.

This immediately gives the first part of part 3), so the desired inequality $\tau( \overline{\im(c^{l'}(Z))}) \leq \prod_{v\in |-l'|}   {t_v \choose (l', E_v)}$.

In the following we want to prove the equality part:

We should prove, that if $w_0$ is an enough generic differential form $w_0 \in H^0(\calO_Z(K+Z))_{reg}$ and we choose for each vertex $u \in |Z|$ numbers $1 \leq a_{u, 1} < \cdots < a_{u, (l', E_u)} \leq t_u$, then $D = \sum_{u \in |Z|, 1 \leq i \leq (l', E_u)} D_{u, a_{u, i}} \in \eca^{l'}(Z)$ is a generic divisor
in the sense, that $c^{l'}(Z)(D)$ is a regular value of the map $c^{l'}(Z)$ and the point $c^{l'}(Z)(D)$ is a smooth point of $\im(c^{l'}(Z))$.

Let's have a generic differential form $w_0 \in H^0(\calO_Z(K+Z))_{reg}$ and an anallytically open subset $w_0 \in U \subset H^0(\calO_Z(K+Z))_{reg}$, such that all differential forms in $U$ are generic, 
and we have holomorphic functions $D_{v, i}: U \to \eca^{-E_v^*}(Z) | v \in |Z|, 1 \leq i \leq t_v$, such that $D_{v, i}(w_0) = D_{v, i}$ and for $w \in U$ one has
$|w| = \sum_{v \in |Z|, 1 \leq i \leq t_v} D_{v, i}(w)$.

Now choose for each vertex $u \in |Z|$ numbers $1 \leq a_{u, 1} < \cdots < a_{u, (l', E_u)} \leq t_u$, then we have to prove that the image of the map
 $D =  \sum_{u \in |Z|, 1 \leq i \leq (l', E_u)} D_{u, a_{u, i}} : U \to \eca^{l'}(Z)$ contains an open subset of $\eca^{l'}(Z)$.

Let's have the map $g:  U \to \eca^{l'}(Z) \to \pic^{l'}(Z)$, we are enough to prove that $g(U)$ contains an open subset of $\im(c^{l'}(Z))$ since the map $\eca^{l'}(Z) \to \pic^{l'}(Z)$
is birational.

Let's denote the contact points $d_{u, i}(w) = D_{u, i}(w) \cap E_u$, where $1 \leq i \leq t_u$.
We claim that we are enough to prove that if $w$ is an enough generic differential form in $U$ and we choose for each vertex $u \in |Z|$ numbers $1 \leq a_{u, 1}, \cdots, a_{u, (l', E_u)} \leq t_u$, then $D(w) | E = \sum_{u \in |Z|, 1 \leq i \leq (l', E_u)} d_{u, a_{u, i}}(w) \in \eca^{l'}(E)$ is a generic divisor in $\eca^{l'}(E)$, so
$D(w)| E $ covers an open subset of $\eca^{l'}(E)$, while $w \in U$.

Indeed assume that this holds and let's denote this open set by $U' \subset \eca^{l'}(E)$.

Let's choose generic points $q_{u, i} \in E_u, 1 \leq i \leq (l', E_u)$, such that $\sum_{u \in |Z|, 1 \leq i \leq (l', E_u) } q_{u, i} \in U'$.

Let's blow up the exceptional divisors at these points and let the new divisors be $E_{u, i}, 1 \leq i \leq (l', E_u)$ and let's denote $l'_{new} = - \sum_{u \in \calv, 1 \leq i \leq (l', E_u)} E_{u, i}^*$ and $Z_{new} = \pi^*(Z)$.

We know that there is a generic divisor $D_{new} \in \eca^{l'_{new}}(Z_{new})$, such that the line bundle $\calO_Z(Z_K + Z - \pi_{*}(D_{new}))$ hasn't got base points at the points $q_{u, i}, u \in \calv, 1 \leq i \leq (l', E_u)$, so there is a divisor $D'$, such that $\calO_Z(Z_K + Z) = \calO_Z(\pi_{*}(D_{new}) + D')$ and $|\pi_{*}(D_{new})| \cap |D'| = \emptyset$.

Let's have a differential form $w' \in H^0(\calO_Z(K+Z))_{reg}$, such that $|w'| = \pi_{*}(D_{new}) + D'$.

On the other hand by the fact that if $w$ is an enough generic differential form $w \in U \subset H^0(\calO_Z(K+Z))_{reg}$, then $D(w) | E = \sum_{u \in |Z|, 1 \leq i \leq (l', E_u)} d_{u, a_{u, i}}(w) \in \eca^{l'}(E)$ is a generic divisor in $\eca^{l'}(E)$ we know that if the points $q_{u, i} \in E_u, 1 \leq i \leq (l', E_u)$ are enough generic, such that $\sum_{u \in |Z|, 1 \leq i \leq (l', E_u) } q_{u, i} \in U'$, then there is a generic differential form $w \in H^0(\calO_Z(K+Z))_{reg}$, such that $d_{u, a_{u, i}}(w) = q_{u, i}$.

Now let's have the divisor of the section $w + t w' \in  H^0(\calO_Z(K+Z))_{reg}$, where $t \in (\bC, 0)$ is enough small and let's denote $D''_t = |w + t w'|$.

We see that if $t \in  (\bC, 0)$ is enough small, then $D''_t$ has got transversal arrows at the points $ q_{u, i}$ and $D_t = \sum_{u \in |Z|, 1 \leq i \leq (l', E_u)} D_{u, a_{u, i}, t} $ is a generic divisor in $\eca^{l'}(Z)$.

It means that we are indeed enough to prove that if $w$ is an enough generic differential form $w \in H^0(\calO_Z(K+Z))_{reg}$ and we choose for each vertex $u \in |Z|$ numbers $1 \leq a_{u, 1}, \cdots, a_{u, (l', E_u)} \leq t_u$, then $D(w) | E = \sum_{u \in |Z|, 1 \leq i \leq (l', E_u)} d_{u, a_{u, i}}(w) \in \eca^{l'}(E)$ is a generic divisor in $\eca^{l'}(E)$, so
$D(w) | E$ covers an open subset of $\eca^{l'}(E)$, while $w \in U$.

We will prove this in two steps, first let's denote the subset of vertices $\calv_1 = (u \in |Z| | Z_u = 1)$ and let's denote $l'_1 = \sum_{u \in \calv_1} - (l', E_u) \cdot E_u^* $, we will prove first that if we choose for each vertex $u \in \calv_1$ numbers $1 \leq a_{u, 1}, \cdots, a_{u, (l', E_u)} \leq t_u$, then $D_1(w) | E = \sum_{u \in \calv_1, 1 \leq i \leq (l', E_u)} d_{u, a_{u, i}}(w) \in \eca^{l'_1}(E)$ is a generic divisor in $\eca^{l'_1}(E)$, so $D_1(w) | E$ covers an open subset of $\eca^{l'_1}(E)$, while $w \in U$. 

Indeed we will prove that the image of the map $D_1 =  \sum_{u \in \calv_1, 1 \leq i \leq (l', E_u)} D_{u, a_{u, i}} : U \to \eca^{l'_1}(Z)$ contains an open subset of $\eca^{l'_1}(Z)$.

Let's have the map $g_1:  U \to \eca^{l'_1}(Z) \to \pic^{l'_1}(Z)$, we are enough to prove that $g_1(U)$ contains an open subset of $\im(c^{l'_1}(Z))$ since the map 
$c^{l'_1}(Z) : \eca^{l'_1}(Z) \to \pic^{l'_1}(Z)$ is birational.

Notice that the dimension of the map $c^{l'_1}(Z) : \eca^{l'_1}(Z) \to \im(c^{l'_1}(Z))$ is $\dim(\im(c^{l'_1}(Z))) = (l'_1, Z)$.

Assume to the contrary that $U$ is a small open neighborhood of $w_0$ and $g_1(U)$ is locally an irreducible analytic subvariety of dimension $d < (l'_1, Z)$, let's denote it by $W$.

For a generic point $q \in W$ let's have $r = h^1(Z, q)$, it means that $W \subset W_r(Z, l'_1)$ and for a generic point $q \in W$ we have $q \in W_r(Z, l'_1) \setminus W_{r+1}(Z, l'_1)$.

Here $W_r(Z, l'_1)$ denotes the Brill-Noether strata $W_r(Z, l'_1) = (\calL \in \pic^{l'_1}(Z) | h^1(Z, \calL) \geq r)$.

Notice now that we have the map $g_1 : U \to W$, and for a generic point $q\in W$ the fiber $g_1^{-1}(q)$ contains differential forms in $H^0(\calO_Z( K+Z))_{reg}$ which have got no pole
on some divisor $D \in (c^{l'_1}(Z))^{-1}(q)$.

From this we get that $\dim(g_1^{-1}(q)) \leq \dim( (c^{l'_1}(Z))^{-1}(q)  ) + h^1(Z, q) = 2r - ( 1 - \chi(Z) - (Z, l'_1))$.

We get the inequality $\dim(U) = 1 - \chi(Z) \leq d + 2r - ( 1 - \chi(Z) - (Z, l'_1))$, which gives that $d \geq 2 h^1(\calO_Z) - (Z, l'_1) - 2r$.

Let's have a generic line bundle $q \in W$, where $q = c^{l'_1}(Z) \left(  \sum_{u \in \calv_1 , 1 \leq i \leq (l', E_u)} D_{u, a_{u, i}}(w) \right)$ and $w \in U$ is a generic point.

Let's have the contact points of the disjoint divisors $D_{u, a_{u, i}}(w)$, $d_{u,  a_{u, i}}(w)$ and assume that $q$ has got a $r_{u,  a_{u, i}}$-simple base point at $d_{u,  a_{u, i}}(w)$.

Since $Z_u = 1$ for each vertex $u \in \calv_1$ we have $r_{u,  a_{u, i}} = 1$ if $q$ has got a base point at $d_{u,  a_{u, i}}(w)$, and $r_{u,  a_{u, i}} = 0$ if $q$ hasn't got a base point at $d_{u,  a_{u, i}}(w)$.

Let's blow up $\tX$ $r_{u,  a_{u, i}}$ times along the divisors $D_{u, a_{u, i}}(w) | u \in \calv_1$ , and let's denote the new singularity by $\tX_{new}$ and the new exceptional divisors
by $E_{u, i, new}$ where $u \in \calv_1, 1 \leq i \leq (l'_1, E_u)$ and $r_{u,  a_{u, i}}= 1$.

Let's denote $Z_{new} = \pi^*(Z) - \sum_{u \in  \calv_1, 1 \leq i \leq (l'_1, E_u), r_{u,  a_{u, i}} = 1} E_{u, i, new}$, notice that $Z_{new}$ is the same cycle as $Z$, just on the blown up singularity.

Let's denote $l'_{1, new } = \sum_{u \in \calv_1, 1 \leq i \leq (l'_1, E_u), r_{u,  a_{u, i}} = 0} E_{u}^*$ and $r' = \sum_{u \in \calv_1, 1 \leq i \leq (l'_1, E_u)} r_{u,  a_{u, i}}$.

Now let's look at the line bundle $q^* =  \pi^*(q)  \otimes \calO_{Z_{new}}(- \sum_{u \in \calv_1, 1 \leq i \leq (l'_1, E_u),  r_{u,  a_{u, i}} = 1} E_{u, i, new})$ and notice that $h^1(Z_{new}, q^*) = r + r'$.

On the other hand if $q$ runs over $W$, then the base point locus of the line bundle $q$ moves in a $ r'$-dimensional family.

This means that if $q \in W$ was enough generic, then there is an analytical subvariety $q^* \in W' \subset \pic^{l'_{1, new}}(Z_{new})$, such that for each $\calL \in W'$ we have
$\calL = c^{l'_{1, new}}(Z_{new})(y)$, where $y = \sum_{u \in \calv_1, 1 \leq i \leq (l'_1, E_u), r_{u,  a_{u, i}} = 0} D_{u, a_{u, i}}(w) $ for some $w \in U$, where $w$ has the same base points as $q$  and $y \in \eca^{l'_{new}}(Z_{new})$, and $\dim(W') \geq d - r'$.

Now there are two cases, assume first that $r_{u,  a_{u, i}}= 1$ for all $u \in \calv_1, 1 \leq i \leq (l'_1, E_u)$, this means that $r' = (l'_1, Z)$ and $q^* = \calO_{Z_{new}}$, which means that
$r = h^1(\calO_Z) - (l'_1, Z)$ so we have $d \geq  (Z, l'_1)$, however this contradicts the assumption $d < (l'_1, Z)$.

So we can assume in the following that $q^*$ is not the trivial line bundle and so there are two independent sections $s_1, s_2 \in H^0(Z_{new}, q^*)_{reg}$ such that $|s_1| \cap |s_2| = \emptyset$.

We know that for a generic point $\calL \in W'$ one has $\calL \in W_{r+r'}(Z_{new},l'_{1, new} ) \setminus W_{r+r' +1}(Z_{new},l'_{1, new} ) $ and $q^* \in W_{r+r'}(Z_{new},l'_{1, new} ) \setminus W_{r+r' +1}(Z_{new},l'_{1, new} )$.

This means that $\dim( T_{q^*}( W_{r+r'}(Z_{new},l'_{1, new} ) )) \geq d - r'$.

Now let's recall the following theorem, which is the analouge of the similar theorem in the classical Brill-Noether theory about the Zariski tangent spaces of Brill-Noether Stratas:

\begin{theorem}
Let $\tX$ be an arbitrary resolution of a normal surface singularity and let's have a Chern class $l' \in -S'$ on it and a cycle $Z$, and let's have a line bundle $\calL \in W_{r}(Z, l') \setminus W_{r+1}(Z, l')$.
Let's look at the bilinear map $ \mu : H^0(Z, \calL) \otimes H^0(Z, \calO_Z(K + Z) \otimes \calL^{-1}) \to H^0(\calO_Z(K + Z))$, then we have $T_{\calL}W_{r}(Z, l') = \ker(\im(\mu))$, where we use the identification $H^0(\calO_Z(K + Z)) = H^1(\calO_Z)^*$.
\end{theorem}

Let's use this theorem in our situation in the following:

Let's look at the map $\mu: H^0(Z_{new}, q^*) \otimes H^0(Z_{new}, \calO_{Z_{new}}(K_{new} + Z_{new}) \otimes (q^*)^{-1} ) \to H^0(\calO_{Z_{new}}( K_{new} + Z_{new}))$, we have $T_{q^*}( W_{r+r'}(Z_{new},l'_{1, new} ) ) = \ker(\im(\mu))$.

Notice that on one hand we have $H^0(Z_{new}, q^*)_{reg} \neq \emptyset$ obviously, and on the other hand $H^0(Z_{new}, \calO_{Z_{new}}(K_{new} + Z_{new}) \otimes (q^*)^{-1} )_{reg} \neq \emptyset$.

The latter statement follows from the fact that $\calO_{Z_{new}}( K_{new} + Z_{new}) = \calO_{Z_{new}}(\sum_{u \in \calv_1, 1 \leq i \leq (l'_1, E_u), 1 = r_{u,  a_{u, i}}} E_{u, i, new} + \sum_{u \in |Z|, 1 \leq j \leq t_u | j \neq a_{u, i}} D_{u, j}) $.

Notice also that the line bundle $q^*$ hasn't got any base points on the cycle $Z_{new}$.

From the theorem it follows that we have obviously $\dim( \im(\mu)) \leq h^1(\calO_{Z_{new}}) - d + r'$.
It means that $\dim( \ker (\mu)) \geq l \cdot (r+ r') - (h^1(Z_{new}) - d + r')$, where we have $2 \leq l = h^0(Z_{new} , q^*)$.

We have $h^0(Z_{new}, \calO_{Z_{new}}(K_{new} + Z_{new}) \otimes (q^*)^{-1} )= h^1(Z_{new}, q^*) = r+r'$ and let's have a basis $s_1, s_2, \cdots s_l \in H^0(Z_{new} , q^*)_{reg}$, such that $|s_1| \cap |s_2| = \emptyset$.

For $1 \leq j \leq l$ let's denote $V_j = <s_1, s_2, \cdots, s_j>$ and let's have the map $\mu_j: V_j \otimes H^0(Z_{new}, \calO_{Z_{new}}(K_{new} + Z_{new}) \otimes (q^*)^{-1} ) \to H^0(\calO_{Z_{new}}( K_{new} + Z_{new}))$.

Notice that for $j > 2$ one has $\dim( \ker (\mu_j)) - \dim( \ker (\mu_{j-1})) \leq h^1(Z_{new}, q^*) = r+r'$, this means that we have:

\begin{equation*}
\dim( \ker (\mu_2)) \geq  2 \cdot (r+ r') - (h^1(\calO_{Z_{new}}) - d + r') = d + 2r + r' - h^1(\calO_{Z_{new}}).
\end{equation*}

Now let's have a generic element in $ \ker (\mu_2)$, it is in the form $s_1 \otimes (t_1) - s_2 \otimes (t_2) \in \ker (\mu_2)$, where we have $t_1, t_2  \in H^0(Z_{new}, \calO_{Z_{new}}(K_{new} + Z_{new}) \otimes (q^*)^{-1} )$ and now assume that $t_1 \in H^0(Z_{new}- A , \calO_{ Z_{new} - A}(K_{new} + Z_{new} - A) \otimes (q^*)^{-1} )_{reg}$, where we have $0 \leq A \leq Z_{new}$ , it follows that we have also $t_2 \in  H^0(Z_{new}- A , \calO_{ Z_{new} - A}(K_{new} + Z_{new} - A) \otimes (q^*)^{-1} )_{reg}$.

Now from the base point free pencil trick we know that $\dim( \ker (\mu_2)) =   H^0(Z_{new}- A , \calO_{ Z_{new} - A}(K_{new} + Z_{new} - A) \otimes 2(q^*)^{-1} )$.

Indeed if $s \in H^0(Z_{new}- A , \calO_{ Z_{new} - A}(K_{new} + Z_{new} - A) \otimes 2(q^*)^{-1} ) $, then $\mu_2( s_1 \otimes (s_2 \cdot s) - s_2 \otimes (s_1 \cdot s)) = 0$, so the map $s \to s_1 \otimes (s_2 \cdot s) - s_2 \otimes (s_1 \cdot s)$ gives the identification between $H^0(Z_{new}- A , \calO_{ Z_{new} - A}(K_{new} + Z_{new} - A) \otimes 2(q^*)^{-1} )$ and $\ker (\mu_2)$.

It means that we have $\dim( \ker (\mu_2)) = h^1(Z_{new}- A , 2q^*)$, so we get:

\begin{equation*}
h^1(Z_{new}- A , 2q^*)  \geq  d + 2r + r' - h^1(Z_{new}).
\end{equation*}

Now notice that we have $H^0(Z_{new}- A , 2q^*)_{reg} \neq \emptyset$ and $H^0(Z_{new}- A , \calO_{ Z_{new} - A}(K_{new} + Z_{new} - A) \otimes 2(q^*)^{-1} )_{reg} \neq \emptyset$.

From this we get $H^0(Z_{new}- A , K_{new} + Z_{new}- A)_{reg} \neq \emptyset$, and $Z_{new}- A$ is a cycle on a generic singularity, so we have $h^1(\calO_{Z_{new}- A}) = 1- \chi(Z_{new}- A)$, or in other words $h^0(\calO_{Z_{new}- A}) = 1$.

Now we have the map $(c^{2l'_1}(Z_{new}- A))^{-1}(\calO_{ Z_{new}- A}(2q^*)) \oplus (c^{-Z_{K_{new}} + Z_{new}- 2l'_1}(Z_{new}- A))^{-1}(\calO_{ Z_{new} - A}(K_{new} + Z_{new} - A) \otimes 2(q^*)^{-1}) \to c^{-Z_{K_{new}} + Z_{new}}(Z_{new}- A))^{-1}(\calO_{ Z_{new}- A}( K_{new} + Z_{new}- A ))$, which is birational to its image, so this gives that:

\begin{equation*}
h^0(Z_{new}- A , 2q^*) + h^0(Z_{new}- A , \calO_{ Z_{new} - A}(K_{new} + Z_{new} - A) \otimes 2(q^*)^{-1}) - 1 \leq  h^0(Z_{new}- A , \calO_{Z_{new}- A}( K_{new} + Z_{new}- A))
\end{equation*}

\begin{equation*}
h^0(Z_{new}- A , 2q^*) + h^0(Z_{new}- A , \calO_{ Z_{new} - A}(K_{new} + Z_{new} - A) \otimes 2(q^*)^{-1}) - 1 \leq 1- \chi(Z_{new}- A).
\end{equation*}

On the other hand we have $h^0(Z_{new}- A , 2q^*) - h^0(Z_{new}- A , \calO_{ Z_{new} - A}(K_{new} + Z_{new} - A) \otimes 2(q^*)^{-1}) = \chi(Z_{new}- A) + (2q^*, Z_{new}- A)$.

From these two inequalities we get instantly that:

\begin{equation*}
h^0(Z_{new}- A , \calO_{ Z_{new} - A}(K_{new} + Z_{new} - A) \otimes 2(q^*)^{-1}) \leq  1-\chi(Z_{new}- A) - (q^*, Z_{new}- A).
\end{equation*}

Notice that from the definition of $Z$ we have $-\chi(Z_{new}- A) - (q^*, Z_{new}- A) < -\chi(Z_{new}) - (q^*, Z_{new})$ if $0 < A \leq Z_{new}$.

Indeed let's denote $A = A' + A''$, where $|A'| \subset |Z|$ and $|A''| \cap |Z| = \emptyset$.

If $A' = 0$, then $ (q^*, Z_{new}- A)  =  (q^*, Z_{new})$, so we have to prove $\chi(Z_{new}- A)  > \chi(Z_{new})$, however this is immediate from $H^0(\calO_{Z_{new}}( K_{new} + Z_{new}))_{reg} \neq \emptyset$.

Assume on the other hand that $A' > 0$ and $-\chi(Z_{new}- A) - (q^*, Z_{new}- A) \geq -\chi(Z_{new}) - (q^*, Z_{new})$, in this case we would have $\dim(\im(c^{c_1(q^*)}(Z_{new}))) = \dim(\im(c^{c_1(q^*)}(Z_{new} - A))) + h^1(\calO_{Z_{new}}) -  h^1(\calO_{Z_{new} - A})$ and we woud easily get that $\dim(c^{l'}(Z)) = \dim(c^{l'}(Z - A')) + h^1(\calO_Z) - h^1(\calO_{Z-A'})$ but this is impossible because of the minimality of $Z$ in its definiton.

Now assume in the following first that $0 < A$, in this case we have:

\begin{equation*}
\dim( \ker (\mu_2)) = h^0(Z_{new}- A , \calO_{ Z_{new} - A}(K_{new} + Z_{new} - A) \otimes 2(q^*)^{-1}) <  1-\chi(Z_{new}) - (q^*, Z_{new}).
\end{equation*}

\begin{equation*}
1-\chi(Z_{new}) - (q^*, Z_{new}) > d + 2r + r' - h^1(\calO_{Z_{new}}).
\end{equation*}

\begin{equation*}
1-\chi(Z_{new}) - (l'_1, Z_{new}) + r' > d + 2r + r' - h^1(\calO_{Z_{new}}).
\end{equation*}

\begin{equation*}
d < 2 h^1(\calO_Z) - (Z, l'_1) - 2r.
\end{equation*}

This is a contradiction, so we can assume in the following, that $A = 0$.

In this case we get that $d \leq 2 h^1(\calO_Z) - (Z, l'_1) - 2r$ and so we must have have equality $d = 2 h^1(\calO_Z) - (Z, l'_1) - 2r$ and this can only happen if $h^1(Z_{new} , 2q^*) =  1-\chi(Z_{new}) - (q^*, Z_{new})$, and furthermore the map $(c^{2l'_1}(Z_{new}))^{-1}(\calO_{ Z_{new}}(2q^*)) \oplus (c^{-Z_{K_{new}} + Z_{new}- 2l'_1}(Z_{new}))^{-1}(\calO_{ Z_{new}}(K_{new} + Z_{new}) \otimes 2(q^*)^{-1}) \to c^{-Z_{K_{new}} + Z_{new}}(Z_{new}))^{-1}(\calO_{ Z_{new}}( K_{new} + Z_{new}))$ has to be dominant.

This means that if we have a generic differential form $w \in U$, where $|w| = \sum_{u \in |Z|, 1 \leq i \leq t_u} D_{u, i}(w)$, where the divisors $  D_{u, i}(w)$ are disjoint,
then there are integers $1 \leq b_{u, 1} < \cdots < b_{u, (2c_1(q^*), E_u)} \leq t_u$, such that $2q^* =  \calO_Z( \sum_{u \in |Z|, 1 \leq j \leq  (2q^*, E_u)} D_{u, b_{u, j}}(w))$.

Notice that this means that there are only finitely many possible values for the line bundle $2q^*$ and so for the line bundle $q^*$ too.
This means that if $w \in U$ general the line bundle $q^*(w) = q^*$ is constant.

Now let's denote $(q^*, E_u) = q_u$, we know that $q_u = 0$ if $Z_u \neq 1$ and $q_u \leq (l', E_u)$ for every vertex $u \in |Z|$ and furthermore for every vertex $u \in |Z|$
there are integers $1 \leq b_{u, 1} < \cdots < b_{u, q_u} \leq t_u$, such that $q^*(w) =  \calO_Z(\sum_{u \in |Z|, 1 \leq j \leq q_u} D_{u, b_{u, j}}(w))$.

Let's have the subsets $M \subset ( (u, i) | Z_u = 1, 1 \leq i \leq t_u)$, such that if $M_u = ( (u, i) | (u, i) \in M )$, then $|M_u| \leq (l', E_u)$ and 
$\calO_Z(\sum_{u, i | (u, i) \in M} D_{u, i}(w))$ is a constant line bundle if $w \in U$ and let's denote the family of this subsets by $F$.

Let's look at the minimal elements of the family $F$ according to containment.

We claim first that if $M_1$ and $M_2$ are minimal elements of $F$, then one has $M_1 = M_2$ or $M_1 \cap M_2 = \emptyset$.

Indeed assume in the following that $M_1 \neq M_2$ but $M_1 \cap M_2 \neq \emptyset$ and assume that $ M_1 \setminus M_2 \neq \emptyset$ too
and let's denote $M' = M_1 \cap M_2$.

We know that $q_{M_1} = \calO_Z(\sum_{u, i | (u, i) \in M_1} D_{u, i}(w))$ and $q_{M_2} = \calO_Z(\sum_{u, i | (u, i) \in M_2} D_{u, i}(w))$ are constant line bundles if $x \in U$ and we know
that the maps $(c^{c^1(q_{M_j})}(Z))^{-1}(q_{M_j}) \oplus (c^{-Z_{K} + Z_{new}- c^1(q_{M_j})}(Z))^{-1}(\calO_{ Z}( K + Z) \otimes  q_{M_j}^{-1}) \to (c^{-Z_{K} + Z}(Z))^{-1}(\calO_{ Z}( K + Z ))$ has to be dominant if $j = 1, 2$.

Now assume that the line bundle $q_{M'}(w) = \calO_Z(\sum_{u, i | (u, i) \in M'} D_{u, i}(w))$ is not constant if we move in $c^{-Z_{K} + Z}(Z))^{-1}(\calO_{ Z}( K + Z ))$, this means that 
the line bundle $q_{M'}(w) = \calO_Z(\sum_{u, i | (u, i) \in M'} D_{u, i}(w))$ is not constant if we move in $(c^{c^1(q_{M_1})}(Z))^{-1}(\calO_{ Z}(q_{M_1}))$.

If we fix a divisor $D \in (c^{-Z_{K} + Z_{new}- c^1(q_{M_1})}(Z))^{-1}(\calO_{ Z}( K + Z) \otimes  q_{M_j}^{-1})$ and move in $(c^{c^1(q_{M_1})}(Z))^{-1}(q_{M_1})$ such that the line bundle $q_{M'}(t) = \calO_Z(\sum_{u, i | (u, i) \in M'} D_{u, i}(t))$ changes, then we get that $ \calO_Z(\sum_{u, i | (u, i) \in M_1} D_{u, i}(t) + D) \in (c^{-Z_{K} + Z}(Z))^{-1}(\calO_{ Z}( K + Z ))$ but the line bundle $ \calO_Z(\sum_{u, i | (u, i) \in M_2} D_{u, i}(t))$ changes, which is a contradiction.

This arguement shows that $M' \in F$ too, however this contradicts to the minimality of $M_1$ so we have proved our claim.

We know that for $w \in U$ we have $q^*(w) = \calO_Z( \sum_{u \in |Z|, 1 \leq j \leq q_u} D_{u, b_{u, j}}(w))$, which means that the set $Q = ((u, b_{u, j}) | 1 \leq j \leq q_u)$ is a subset in the family 
$F$, so let's have a minimal element $M \in F$, such that $M \subset Q$.

Notice that because the line bundle $\calO_{ Z}( K + Z )$ hasn't got a base point, if we move around in $ c^{-Z_{K} + Z}(Z))^{-1}(\calO_{ Z}( K + Z ))$ the monodromy is
$1$-transitive on the set $D_{u, i} , 1 \leq i \leq t_u$ on every vertex $u \in |Z|$. 

We claim that for each vertex $u\in |Z|$, $M$ can't contain all the elements $(u, i), 1 \leq i \leq t_u$.

Indeed suppose that we have a vertex $u' \in |Z|$ and $M$ contains all the elements $(u', i), 1 \leq i \leq t_{u'}$.

Now there are two cases, assume first that there exists another vertex $u'' \in |Z|$, such that $M$ doesn't contain all the elements $(u'', i), 1 \leq i \leq t_{u''}$, but $M_{u''} \neq \emptyset$.

Let's have an element $(u'', i_1) \in M_{u''}$ and another element $(u'', i_2) \notin M_{u''}$, we know that the monodromy is $1$-transitive on the set $D_{u'', i} , 1 \leq i \leq t_{u''}$, so 
we get that there is a minimal element $M' \in F$ such that $(u'', i_2) \in M'_{u''}$ and $M'$ contains all the elements $(u', i), 1 \leq i \leq t_{u'}$.

However this is a contradiction because $M \neq M'$ are minimal elements in $F$, but $M \cap M' \neq \emptyset$.

In the other case assume that $q_{M}(w) = \calO_Z(\sum_{u \in Y, 1 \leq i \leq t_u} D_{u, i}(w))$, where $Y \subset |Z|$, however in this case $t_u = (l', E_u)$ for each vertex $u \in Y$,
since $M \subset Q$ and the line bundle $q_{M}(w)$ is constant if $w \in U$.

Notice on the other hand that there are inidices $1 \leq r_{u, 1}, \cdots, r_{u, (l', E_u)} \leq t_u$ for every vertex $u \in |Z|$ such that the line bundle $\calO_Z(\sum_{u \in |Z|, 1 \leq i \leq (l', E_u)}
D_{u, r_{u, i}}(w))$ covers an open set in $\im(c^{l'}(Z))$, which has got dimension $(l', Z)$.

On the other hand the line bundle $q_{M}(w) = \sum_{u \in Y, 1 \leq i \leq (l', E_u)} D_{u, i}(w)$ is constant and the line bundle  $\calO_Z(\sum_{u \in |Z| \setminus Y, 1 \leq i \leq (l', E_u)} D_{u, r_{u, i}}(w))$ covers a set of dimension at most $(l', Z_{|Z| \setminus Y}) < (l', Z)$ which is a contradiction.

So we have proved our claim that for each vertex $u\in |Z|$ $M$ can't contain all the elements $(u, i), 1 \leq i \leq t_u$.

Next we claim that for each vertex $u\in |Z|$ one has $|M_u| \leq 1$, indeed assume to the contrary that $u' \in |Z|$ and $|M_{u'}| \geq 2$, we already know however that $|M_{u'}| < t_{u'}$.

So we can assume now that $(u', 1), (u', 2) \in M_{u'}$ and $ (u', 3) \notin M_{u'}$.

Let's blow up now $E_{u'}$ in a generic point $p$, let's denote the new singularity by $\tX_{new}$, the cycle $Z_{new} = \pi^*(Z) - E_{new}$ and let's look at the line bundle 
$\calO_{Z_{new}}(Z_{new} + K_{new} - E_{new})$.

We know that $(l', E_{u'}) \geq 2$, so it follows from the minimality of the cycle $Z$ that $e_Z(u') \geq 3$.

It means by lemma\ref{hyper2} that the line bundle $\calO_{Z_{new}}(Z_{new} + K_{new} - E_{new})$ hasn't got a base point on the exceptional divisor $E_{u'}$.

It means that if we move around in $ (c^{-Z_{K} + Z}(Z))^{-1}(\calO_{ Z}( K + Z ))$, then the monodromy is $2$-transitive on the set $D_{u', i} , 1 \leq i \leq t_{u'}$.

We give the sketch of this easy arguement:

Let's have two indices $1 \leq i_1, i_2 \leq t_{u'}$, we know that if $w \in H^0(\calO_{Z}(K + Z))$ is generic then the contact points $d_{u', 1}(w) , d_{u', 2}(w)$
are two generic points in $E_{u'}$ and also $d_{u', i_1}(w) , d_{u', i_2}(w)$ are two generic points in $E_{u'}$.
It means that we can have two generic points $p, q \in E_{u'}$ and two generic diferential forms $w_1, w_2 \in H^0(\calO_{Z}(K + Z))_{reg}$ such that $d_{u', 1}(w_1) = p, d_{u', 2}(w_1) = q$ and $d_{u', i_1}(w_2) = p, d_{u', i_2}(w_2) = q$. 
The family of differential forms $t \cdot w_1 + (1-t)w_2, t \in \bC$ gives the desired two-transitivity of the monodromy.

However this means that there is a minimal element $M' \in F$ such that $(u', 1), (u', 3) \in M'_{u'}$, but this is a contradiction because of $M' \cap M \neq \emptyset$ and $M' \neq M$.

Thus we have concluded that for each vertex $u\in |Z|$ one has $|M_u| \leq 1$.

Because the monodromy is $1$ transitive we obviously know that there are at least two vertices $u \in |Z|$, such that $|M_u| = 1$, let's denote two such vertices by $u', u''$.
Let's denote the subset of $|Z|$ consisting of vertices $u \in |Z|$, such that $|M_u| = 1$ by $G$, we can assume that $(u', 1), (u'', 1) \in M$.

Now let's blow up the singularity at $d_{u', 1}(w)$ for some generic $w \in U$,  let's denote the new singularity by $\tX_{new}$, the cycle $Z_{new} = \pi^*(Z) - E_{new}$ and let's look at the line bundle $\calO_{Z_{new}}(Z_{new} + K_{new} - E_{new})$, we claim that it has got a base point at $d_{u'', 1}(w)$.

Indeed assume that it hasn't got a base point at $d_{u'', 1}(w)$, then we can find two generic points $p \in E_{u'}, q \in E_{u''}$ and $w \in U$, such that $d_{u', 1}(w) = p, d_{u'', 1}(w) = q$.

On the other hand if $p \in E_{u'}, q \in E_{u''}$ are generic, then we can find $w' \in H^0(\calO_Z(Z + K))$ differential form, such that $p, q \in |w'|$ and there are 
integers $1 \leq j_u \leq t_u, u \in G$, such that $d_{u', j_{u'}}(w') = p, d_{u'', j_{u''}}(w')= q$ and $ \calO_Z(\sum_{u \in G} D_{u, j_{u}}(w')) \neq q_M$.

Now let's have the differential form $t(w') + (1-t) w \in U, t \in \bC$, we get that $q_M(t \cdot w + (1-t) \cdot  w') \neq q_M(w)$ if $t \in \bC$ is enough close to $0$, which is a contradiciton.

Notice that we got that $d_{u'', 1}$ is a base point of the line bundle $\calO_Z(Z + K - d_{u', 1})$ which means, that $h^0(\calO_Z(Z + K - d_{u', 1}))) = h^0(\calO_Z(Z + K - d_{u', 1} - d_{u'', 1})))$ and we get that $h^0(\calO_Z(d_{u', 1} + d_{u'', 1})) = 2$.

Notice however that $(l', E_u') \geq 1$ and $(l', E_u) \geq 1$ so by the minimality of the cycle $Z$ we get that $e_Z(u', u'') \geq 3$, however this contradicts lemma\ref{hyper1}.

 This contradiction proves our claim that if we denote the subset of vertices $\calv_1 = (u \in |Z| | Z_u = 1)$ and $l'_1 = \sum_{u \in \calv_1} -(l', E_u) \cdot E_u^* $, then if we choose
for each vertex $u \in \calv_1$ numbers $1 \leq a_{u, 1}, \cdots, a_{u, (l', E_u)} \leq t_u$, then $D_1(w) | E = \sum_{u \in \calv_1, 1 \leq i \leq (l', E_u)} d_{u, a_{u, i}}(w)  \in \eca^{l'_1}(E)$ is a generic divisor in $\eca^{l'_1}(E)$, so $D_1(w)$ covers an open subset of $\eca^{l'_1}(E)$, while $w \in U$. 

On the other hand we should prove that if $w$ is an enough generic differential form $w \in H^0(\calO_Z(K+Z))_{reg}$ and we choose for each vertex $u \in |Z|$ numbers $1 \leq a_{u, 1}, \cdots, a_{u, (l', E_u)} \leq t_u$, then $D(w) | E = \sum_{u \in |Z|, 1 \leq i \leq (l', E_u)} d_{u, a_{u, i}}(w) \in \eca^{l'}(E)$ is a generic divisor in $\eca^{l'}(E)$, so $D(w) | E$ covers an open subset of $\eca^{l'}(E)$, while $w \in U$.

However the second statement follows trivially from the first statement using our lemma\ref{ind} many times.

Indeed let's have the pairs $(u, a_{u, i}), u \notin \calv_1$ and let's order them in an arbitrary order, so let's have $r = |((u, a_{u, i}) | u \notin \calv_1)|$ and for
$1 \leq j \leq r$ let's denote the $j$-th pair by $(u_j, a_j)$ and let's denote $l'_j = -\sum_{1 \leq i \leq j} E_{u_i}^*$

We should prove inductively on the paramater $j$, that $D_{j}(w) | E = \sum_{u \in \calv_1, 1 \leq i \leq (l', E_u)} d_{u, a_{u, i}}(w) + \sum_{1 \leq i \leq j} d_{u_i, a_i}(w) \in \eca^{l'_{j}}(E)$
is a generic divisor in $\eca^{l'_{j}}(E)$, so $D_{j}(w) | E$ covers an open subset of $\eca^{l'_{j}}(E)$, while $w \in U$.
We know the statement for $j = 0$, so let's see the induction step.

We know from the induction hypothesis that $D_{j-1}(w) | E = \sum_{u \in \calv_1, 1 \leq i \leq (l', E_u)} d_{u, a_{u, i}}(w) + \sum_{1 \leq i \leq j-1} d_{u_i, a_i}(w)$
is a generic divisor in $\eca^{l'_{j-1}}(E)$, and the induction hypothesis follows trivially from lemma\ref{ind}.

Indeed for each index $1 \leq i \leq j-1$ let's blow $E_{u_i}$ in a generic point $q_i \in E_{u_i}$ and let the new divisors be $E_{1}, \cdots, E_{j-1}$ and let's denote
$l = \sum_{1 \leq i \leq j-1} E_i$ and $Z_{new} = \pi^*(Z) - l$.

Notice that if $q_j \in E_{u_j}$ is a generic point then there exists a section $s \in H^0(\calO_Z(K+Z))$ which goes through $q_1, \cdots, q_j$ which means that $H^0(\calO_{Z_{new}}(K_{new} + Z_{new} - l))_{reg} \neq \emptyset$ and the dimension of the image of the map $H^0(\calO_{Z_{new}}( K_{new} + Z_{new} - l)) \to H^0(\calO_{E_{u_j}}(K_{new} + Z_{new} - l))$ is bigger than $1$.

By lemma\ref{ind} this means that the line bundle $\calO_{Z_{new}}(K_{new} + Z_{new} - l)$ hasn't got a base point on the regular part of $E_{u_j}$.
This indeed yields that $D_{j}(w) | E = \sum_{u \in \calv_1, 1 \leq i \leq (l', E_u)} d_{u, a_{u, i}}(w) + \sum_{1 \leq i \leq j} d_{u_i, a_i}(w) $
is a generic divisor in $\eca^{l'_{j}}(E)$ and we are done.

\end{proof}

Let's consider in the following the case of an arbitrary normal surface singularity $\tX$ with resolution graph $\mathcal{T}$.
In this case the line bundles $\calO_Z(K+Z)$ may have several unexpected base points, however we can prove easily the upper bound part analouge of the previous theorem:

\begin{theorem*}
Let $\mathcal{T}$ be an arbitrary resolution graph and $\tX$ a singularity corresponding to it.
Let's have a Chern class $l' \in -S'$ and an integer effective cycle $Z \geq E$, such that $Z = C_{min}(Z, l')$, in particular we know that the Abel map $ \eca^{l'}(Z) \to \im(c^{l'}(Z))$ is birational.

For an arbitrary vertex $v \in \calv $ let's denote $t_v = (- Z_K+Z, E_v)$, with this notation we have got $\tau( \overline{\im(c^{l'}(Z))}) < \prod_{v\in |l'|_{*}}   {t_v \choose (l', E_v)}$.
\end{theorem*}
\begin{remark}
In the general case it can happen that $\tau( \overline{\im(c^{l'}(Z))}) = 0$, so the dual projective variety of the projective clousure $\overline{\im(c^{l'}(Z))}$ has got dimension less 
than $h^1(\calO_Z)-1$.
\end{remark}

\begin{proof}

Let's denote in the following $\tau = \tau( \overline{\im(c^{l'}(Z))})$.

Let's have a generic differential form $w \in H^0(\calO_Z(K + Z)) = H^1(\calO_Z)^*$, and suppose that $p_1, \cdots, p_{\tau} \in \im(c^{l'}(Z))$ are different generic smooth points of $\im(c^{l'}(Z))$ such that $w$ vanishes on $T_{p_i}(\im(c^{l'}(Z)))$.

Let's denote the set of base points of the line bundle $\calO_Z(K + Z)$ by $B$, for a generic section $w \in H^0(\calO_Z(K + Z))_{reg}$ we can write 
$|w| = \sum_{v \in |Z|, 1 \leq i \leq a_v} D_{v, i} + D'$ where $a_v \leq t_v$ and $D_{v, i}$ are disjoint transversal cuts on the exceptional divisor $E_v$, $D_{v, i} \cap B = \emptyset$,
$D' \cap E \subset B$.

We can also assume that $w$ is such generic that the points $p_i \in \im(c^{l'}(Z))$ satisfy $\dim( (c^{l'}(Z))^{-1}(p_i)) = 0$ and if we denote $D_i = (c^{l'}(Z))^{-1}(p_i)$, then
$D_i$ consists of $(l', E)$ disjoint transversal cuts along the smooth part of $E$.

If $w$ is enough generic we can assume that the only $0$-dimensional divisor $D_i \in |p_i|$ satisfies $D_i \cap B = \emptyset$ if $ 1 \leq i \leq \tau$.

We can also assume that the Abel map is submersion in the points $D_i \in \eca^{l'}(Z)$ so $T_{p_i}( \im(c^{l'}(Z)) = \im(T_{D_i}(c^{l'}(Z)))$

We know that the differential form $w$ vanishes on the subspace $\im(T_{D_i}(c^{l'}(Z)))$, so the differential form $w$ hasn't got a pole on the divisors $D_i, 1 \leq i \leq \tau$.

On the other hand by lemma\ref{cut} we know that $D_i = \sum_{v \in |Z|, 1 \leq j \leq (l', E_v)} D_{v, b_j}$ where $1 \leq b_1< \cdots < b_{(l', E_v)} \leq a_v$ are different indices.

This means that $\tau \leq \prod_{v \in |Z|} {a_v \choose (l', E_v)} \leq \prod_{v\in |Z|} {t_v \choose (l', E_v)}$ which proves our theorem.
\end{proof}


\begin{thebibliography}{30}

\bibitem[AC]{AC} Enrico Arbarello, Mauricio Cornalba, Phillip A. Griffith, J. Harris:
Geometry of Algebraic curves

\bibitem[A62]{Artin62} Artin, M.:
Some numerical criteria for contractibility of curves on algebraic surfaces.
{\em  Amer. J. of Math.}, {\bf 84}, 485-496, 1962.

\bibitem[A66]{Artin66} Artin, M.:
On isolated rational singularities of surfaces.
{\em Amer. J. of Math.}, {\bf 88}, 129-136, 1966.



\bibitem[BN10]{BN} Braun, G. and N\'emethi, A.:
Surgery formula for Seiberg--Witten invariants of negative definite plumbed 3--manifolds,
{ \em J. f\"ur die reine und ang. Math.} {\bf 638} (2010), 189--208.

\bibitem[BN07]{BNnewt} Braun, G. and N\'emethi, A.:
Invariants of Newton non-degenerate surface singularities, {\em Compositio Math.} {\bf 143} (2007), 1003--1036.

\bibitem[BV99]{BV} Brion, M. and Vergne, M.: Arrangement of hyperplanes. I: Rational functions and the Jeffrey--Kirwan residue,
{\em Ann. Sci. l’École Norm. Sup.} {\bf 32} (1999), no. 5, 715--741.

\bibitem[CDGZ04]{CDGPs} Campillo, A.,  Delgado, F. and Gusein-Zade, S. M.:
Poincar\'e series of a rational surface singularity, {\em Invent. Math.} {\bf 155} (2004),
no. 1, 41--53.

\bibitem[CDGZ08]{CDGEq}  Campillo, A.,  Delgado, F. and Gusein-Zade, S. M.:
Universal abelian covers of rational
surface singularities and multi-index filtrations,
{\em Funk. Anal. i Prilozhen.} {\bf 42} (2008), no. 2, 3--10.

\bibitem[CHR03]{CHR} Cutkosky, S. D.,  Herzog, J. and Reguera, A.:
Poincar\'e series of resolutions of surface singularities,
{\em Trans. of the AMS} {\bf 356} (2003), no. 5, 1833--1874.

\bibitem[EN85]{EN} Eisenbud, D. and Neumann, W.: Three--dimensional link theory and invariants of plane curve singularities,
{\em Princeton Univ. Press} (1985).

\bibitem[GS99]{GS} Gompf, R.E. and Stipsicz, A.:  An introduction to
$4$--manifolds and Kirby calculus, {\em Graduate Studies in
Mathematics} {\bf 20} (1999), Amer. Math. Soc.

\bibitem[GR70]{GR} Grauert, H., Remmert, R.:
Coherent analytic sheaves, Grundlehren der mathematischen Wissenschaften, 1984 Springer 

\bibitem[GrRie70]{GrRie} Grauert, H. and Riemenschneider, O.: Verschwindungss\"atze f\"ur analytische
kohomologiegruppen auf komplexen R\"aumen, {\it Inventiones math.} {\bf 11} (1970), 263-292.



\bibitem[Ha77]{hartshorne}  Hartshorne, R.:  Algebraic Geometry, Graduate Texts
in Mathematics {\bf 52}, Springer-Verlag, 1977.




\bibitem[L13]{LPhd} L\'aszl\'o, T.: Lattice cohomology and Seiberg--Witten invariants of normal surface singularities, PhD. thesis,
Central European University, Budapest, 2013.

\bibitem[LN14]{LN} L\'aszl\'o, T. and N\'emethi, A.: Ehrhart theory of polytopes and Seiberg-Witten invariants of plumbed 3--manifolds,
{\em Geometry and Topology} {\bf 18} (2014), no. 2, 717--778.

\bibitem[LN15]{LNRed} L\'aszl\'o, T. and N\'emethi, A.: Reduction theorem for lattice cohomology,
{\em Int Math Res Notices} {\bf 11} (2015), 2938--2985.



\bibitem[La72]{Laufer72} Laufer, H.B.: On rational singularities,
{\em Amer. J. of Math.}, {\bf 94}, 597-608, 1972.

\bibitem[La73]{LaDef1}
Laufer, Henry B.: Deformations of Resolutions of Two-Dimensional Singularities,
 {\it Rice Institute Pamphlet - Rice University Studies} {\bf 59} no. 1 (1973), 53--96.
  Rice University: http://hdl.handle.net/1911/63103.

\bibitem[Les96]{Lescop} Lescop, C.: Global surgery formula for the Casson--Walker
invariant, {\em Ann. of Math. Studies}  {\bf 140}, Princeton Univ. Press, 1996.

\bibitem[Lim00]{Lim} Lim, Y.: Seiberg--Witten invariants for 3--manifolds in the case $b_1=0$ or $1$,
{\em Pacific J. of Math.} {\bf 195} (2000), no. 1, 179--204.

\bibitem[Li69]{Lipman} Lipman, J.: Rational singularities, with applications to algebraic surfaces
and unique factorization, Inst. Hautes \'Etudes Sci. Publ. Math. {\bf 36} (1969), 195-279.






\bibitem[Kl??]{Kl} Kleiman, St. L.: The Picard scheme,



\bibitem[N99b]{Nfive} N\'emethi, A.: Five lectures on normal surface singularities,
lectures at the Summer School in {\em Low dimensional topology} Budapest,
Hungary, 1998; Bolyai Society Math. Studies {\bf 8} (1999), 269--351.


\bibitem[NO09]{NOk} N\'emethi, A. and Okuma, T.:
On the Casson invariant conjecture of Neumann--Wahl, {\em Journal of Algebraic Geometry}
{\bf 18} (2009), 135--149.



\bibitem[NNI]{NNI} Nagy, J., N\'emethi, A.:
The Abel map for surface singularities  I. Generalities and  examples,
Matematiche Annalen


\bibitem[NNII]{NNII} Nagy, J., N\'emethi, A.:
The Abel map for surface singularities  II. Generic analytic structure,
Advances In Mathematics


 \bibitem[NNAD]{NNAD} Nagy, J., N\'emethi, A.:
The dimension of the image of the Abel map associated with normal surface singularities 
 arXiv:1909.07023

\bibitem[NNM]{NNM} Nagy, J., N\'emethi, A.:
The multiplicity of generic normal surface singularities
arxiv: 2005.10867 


\bibitem[NR]{R} Nagy, J:
Invariants of relatively generic structures on normal surface singularities,
arXiv:1910.03275

\bibitem[H]{H} Nagy, J:
Hyperelyptic involutions on generic normal surface singularities,
arXiv:2006.05869 

\bibitem[NW90]{NWCasson}
W. Neumann and J.~Wahl, Casson invariants of links of singularities,
Comment. Math. Helvetici {\bf 65} (1990), 58--78.

\bibitem[NW05]{NWsq}
Neumann, W. and Wahl, J.: Complete intersection singularities of splice type as universal abelian covers,
{\em Geom. Topol.} {\bf 9} (2005), 699--755.

\bibitem[NW10]{NWECTh}
W.~D. Neumann and J.~Wahl,
\emph{ The End Curve Theorem for normal complex surface singularities},  J. Eur. Math. Soc. {\bf 12} (2010), 471--503.


\bibitem[Nic04]{Nic04} Nicolaescu, L.: Seiberg--Witten invariants of rational homology $3$--spheres,
{\em Comm. in Cont. Math.} {\bf 6} no. 6 (2004), 833--866.



\bibitem[O04]{OkumaRat} Okuma, T.: Universal abelian covers of rational surface singularities,
{\it Journal of London Mathematical Society} {\bf 70}(2) (2004), 307-324.

\bibitem[O08]{Ok} Okuma, T.: The geometric genus of splice--quotient singularities,
{\em Trans. Amer. Math. Soc.} {\bf 360} 12 (2008), 6643--6659.

\bibitem[O10]{OECTh}
Okuma, T.:
\emph{Another proof of the end curve theorem for normal surface singularities},
J. Math. Soc. Japan {\bf 62}, Number 1 (2010), 1--11.




\bibitem[Ra72]{Ram} Ramanujam, C.P.: Remarks on Kodaira vanishing theorem, {\it J. Indian Math. Soc.} {\bf 36}
(1972), 41-51.


\bibitem[Re97]{MR}  Reid, M.: Chapters on Algebraic Surfaces.
In: Complex Algebraic Geometry,
IAS/Park City Mathematical Series,  Volume {\bf 3}  (J. Koll\'ar editor),
3-159, 1997.




\end{thebibliography}
\end{document}